\documentclass[11pt]{article}
\usepackage[pagewise]{lineno}\linenumbers
\usepackage[T1]{fontenc}
\usepackage{latexsym,amssymb,amsmath,amsfonts,amsthm}
\usepackage{graphicx}
\usepackage{epstopdf}
\usepackage{float}
\usepackage{subfigure}
\usepackage{caption}
\usepackage{amsmath}
\usepackage{mathtools}
\numberwithin{figure}{section}
\numberwithin{table}{section}
\usepackage{placeins}
\usepackage{color}
\usepackage{indentfirst}
\numberwithin{equation}{section}
\newcommand{\R}{{\mathbb R}}
\newcommand{\Z}{{\mathbb Z}}
\newcommand{\N}{{\mathbb N}}
\newcommand{\C}{{\mathbb C}}
\newcommand{\s}{{\mathbb S}}

\newcommand{\en}{\end{eqnarray}}
\newcommand{\enn}{\end{eqnarray*}}

\newcommand{\curl}{{\rm curl\,}}

\newcommand{\grad}{{\rm grad\,}}
\newcommand{\divv}{{\rm div\,}}

\newtheorem{theorem}{Theorem}[section]
\newtheorem{lemma}[theorem]{Lemma}
\newtheorem{corollary}[theorem]{Corollary}
\newtheorem{definition}[theorem]{Definition}
\newtheorem{remark}[theorem]{Remark}

\definecolor{rot}{rgb}{1.000,0.000,0.000}
\definecolor{blau}{rgb}{0,0,1}

\begin{document}
\captionsetup[figure]{labelfont={bf},name={Fig.},labelsep=period}
	
	\title{Factorization method for inverse time-harmonic elastic scattering with a single plane wave}

	\author{Guanqiu Ma\footnotemark[1]\  \and Guanghui Hu\footnotemark[2]\ }
	
	\date{}
	\maketitle
	
	\renewcommand{\thefootnote}{\fnsymbol{footnote}}
	\footnotetext[1]{Department of Applied Mathematics, Beijing Computational Science Research Center, Beijing 100193, P. R. China (guanqium@csrc.ac.cn).}
	\footnotetext[2]{School of Mathematical Sciences and LPMC, Nankai Universtiy, Tianjin 300071, China (ghhu@nankai.edu.cn, corresponding author).}

	\renewcommand{\thefootnote}{\arabic{footnote}}

	\begin{abstract}
	This paper is concerned with the factorization method with a single far-field pattern to recover an arbitrary convex polygonal scatterer/source in linear elasticity. 
The approach also applies to the compressional (resp. shear) part of the far-field pattern excited by a single compressional (resp. shear) plane wave.  
The one-wave factorization is based on the scattering data for a priori given testing scatterers. It can be regarded as a domain-defined sampling method and does not require forward solvers.  We derive the spectral system of the far-field operator for  rigid disks and show that, using testing disks, the one-wave factorization method can be justified independently of the classical factorization method.


\vspace{.2in} {\bf Keywords}: factorization method, inverse elastic scattering, linear elasticity,  single far-field pattern, polygonal scatterers, corner scattering.
	\end{abstract}

	
    \section{Introduction}
	
The purpose of inverse time-harmonic elastic scattering is to recover the position, shape and physical properties of an elastic body by using information of the scattered wave generated by time-harmonic plane and point source waves. We refer to \cite{Ku1979, MC2005, Habib2015} for a comprehensive introduction of mathematical theory and inverse problems in linear elasticity.

Over the last twenty years,
 sampling-type methods have attracted much attention, because forward solvers and good initial approximations of the target are no longer required,  in contrast with the iterative approaches.  The multi-wave sampling methods do not require a priori information on physical and geometrical properties of the scatterer, but usually need  far-field data for a large number of incident waves. Here we give an incomplete list of the applications to the Navier equation, including linear sampling and factorization methods \cite{Arens, kress2002, Cha2003, Cha2007, HuKirsch2012}, singular source method \cite{Das2008}, orthogonal/direct sampling method \cite{Ji2018}, enclosure method\cite{GS2012} and the Reverse time-migration method in the frequency domain \cite{HC}. On the other hand, there also exists the so-called one-wave sampling methods, which are usually designed to test the analytic extensibility of the scattered field; see the monograph \cite[Chapter 15]{NP2013} for detailed discussions on scalar equations, for instance, range test  and no-response test \cite{K2003, Lin2021, L2003} and enclosure method \cite{I1999, II2009}. The one-wave method requires only a single far-field pattern or one-pair Cauchy data, but one must pre-assume the absence of an analytical continuation across the scattering interface. 

If a single far-field pattern is available only, the inverse scattering problems become severely ill-posed and thus more challenging. 
 This paper is concerned with the one-wave factorization method for recovering a convex rigid elastic body of polygonal type from a single elastic far-field pattern.  Such a method 
 was earlier discussed in \cite{el-hu19} for inverse elastic scattering from rigid polygonal bodies but without too much details. It  
 is closet to the extended linear sampling method \cite{sun2019, S2018} and the one-wave range test method \cite{K2003}.
In the authors' previous work \cite{M-H}, the one-wave factorization method for the Helmholtz equation was rigorously established with the help of corner scattering theory. The connections to the range test and extended linear sampling were also discussed there. 
The one-wave factorization method is a both data-driven and model driven method, and could lead to an explicit characterization of  an arbitrary convex scatterer of polygonal type if the testing scatterers are chosen as disks. 
In this sense, it inherits merits of the classical factorization method for precisely characterizing targets \cite{K2008} but restricted to convex polygonal scatterers/sources. The purpose of this paper is to generalize the mathematical theory of \cite{M-H} to the Navier equation.   The following items  can be considered as  complementary contributions to the previous work \cite{el-hu19}: i) One-wave factorization method using only compressional or shear waves in linear elasticity; ii) Explicit expression of the spectral data for elastic far-field operators corresponding to rigid disks and a straightforward verification of the one-wave factorization method by using testing disks.

 This paper is organized as follows. In Section \ref{sec:2}, we introduce basic concepts of the direct and inverse elastic scattering problems. In Section \ref{sec:3}, the multi-wave factorization method for recovering a rigid scatterers will be briefly reviewed. In Section \ref{sec:4}, we present a rigorous justification of the one-wave method by combining the classical factorization method and elastic corner scattering theory. Explicit examples by using testing  disks will be presented in Section \ref{sec:5}, including 
 derivation of an eigensystem of the far-field operator for a rigid disk.
 Finally, we describe our imaging schemes in Section \ref{sec6}. 
	
	\section{Preliminaries}\label{sec:2}
	In this paper, we will consider the scattering of elastic waves in two-dimensional space $\mathbb{R}^2$. Let $D \subset \mathbb{R}^2$ be a bounded rigid elastic body with connected exterior $D^c \coloneqq \mathbb{R}^2 \backslash \bar{D}$. Let $D^c$ be filled with a homogeneous and isotropic elastic medium. Suppose that a time-harmonic elastic plane wave of the form 
	\begin{equation}\label{2.1}
		u^{i}\left(x; d, c_p, c_s \right) = c_p d \text{e}^{ik_px\cdot d} + c_s d^{\bot} \text{e}^{ik_sx\cdot d}, \quad c_p,c_s \in \mathbb{C}, \quad |c_p|+|c_s|\neq 0,
	\end{equation} 
	is incident onto the scatterer $D$. Here $d = (\cos \theta_d,\sin \theta_d )^T, \theta_d \in \left[0,2\pi\right)$ is the incident direction; $d^{\bot} \coloneqq (-\cos \theta_d, \sin \theta_d)^T$ is a vector orthogonal to $d$; $\omega >0$ is the frequency;  $ k_p \coloneqq \omega / \sqrt{\lambda+2\mu}$ and $k_s \coloneqq \omega / \sqrt{\mu}$ are the compressional and shear wave numbers, respectively. Note that for simplicity the density of the background medium has been normalized to be one and the Lame constants 
$\lambda$ and $\mu$ satisfy $\mu > 0$ and $\lambda +2 \mu > 0$ in two dimensions. The propagation of time-harmonic elastic waves in $D^c$ is governed by the Navier equation (or system)
	\begin{equation}\label{lame-eqn}
		\Delta^{\ast}u + \omega^2 u \coloneqq \mu \vartriangle u + \left( \lambda +\mu \right) \triangledown \left( \triangledown \cdot u\right) + \omega ^2 u = 0 \quad \text{in} \quad D^c, \quad u = \left( u_1, u_2\right)^T,
	\end{equation}
	where $u = u^{i} + u^{s}$ donotes the total displacement field. By Hodge decomposition, any solution $u$ to \eqref{lame-eqn} can be decomposed into the from
	\begin{equation}\label{Hodgedecom}
		u = u_p + u_s, \quad u_p \coloneqq -\frac{1}{k^2_p} \ \text{grad} \ \text{div} \ u, \quad u_s \coloneqq \frac{1}{k^2_s} \ \text{curl} \ \overrightarrow{\text{curl}} \ u, 
	\end{equation}
	where $u_p$ and $u_s$ are called compressional and shear waves respectively. Note that in \eqref{Hodgedecom} the two curl operators are defined as 
	\begin{equation}
		\overrightarrow{\text{curl}} \ u \coloneqq \partial_2 u_1 - \partial_1 u_2, \quad \text{curl} \ f = \left(-\partial_2 f,\partial_1 f\right)^T,
	\end{equation}
	satisfying the relation
	\begin{equation}
		\text{curl} \  \overrightarrow{\text{curl}} \ u = -\Delta u + \text{grad} \ \text{div} \  u.
	\end{equation}
	Moreover, $u_{\alpha}\left(\alpha=p,s\right)$ satisfies the vector Helmholtz equations $\left(\Delta + k^2_{\alpha}\right) u_{\alpha}=0$ and $\overrightarrow{\text{curl}} \ u_p = \text{div} \ u_s = 0$ in $D^c$. 
	In this paper, we require $u^{s}$ to fulfill the Kupradze radiation condition 
	\begin{equation}\label{Kupr-cd}
		\partial_r u^{s}_{\alpha} - ik_{\alpha}u^{s}_{\alpha} = o\left(r^{-\frac{1}{2}}\right) \quad \text{as} \quad r=|x| \to \infty, \quad \alpha = p, s,
	\end{equation}
	uniformly in all directions $\hat{x} = x/|x|$ on the unit circle $\mathbb{S} \coloneqq \left\{x \in \mathbb{R}^2 : |x| = 1\right\}.$ 
It is well known that the direct scattering problem admits one solution $u\in C^2(\R^2\backslash\overline{D})\cap C^1(\R^2\backslash D)$ if $\partial D$ is of $C^2$-smooth (see \cite{Ku1979}) and $u\in (H^1_{loc}(\R^2\backslash\overline{D}))^2$ if $\partial D$ is Lipschitz (see e.g., \cite{Bao2018, Li2016}).
	
This paper is concerned with the inverse scattering problem of recovering $\partial D$ from the information of the far-field pattern of a single incoming plane wave. 
The compressional and shear parts $u^{s}_{\alpha} \left( \alpha =p,s \right)$ of the radiating solution $u^{s}$ admit an asymptotic behavior of the form \cite{el-hu19,PGC1993}
	\begin{equation}
		\begin{split}
			u^{s}_{p}\left( x\right) &= \frac{\text{e}^{ik_pr}}{\sqrt{r}} \left\{ u^{\infty}_p \left( \hat{x} \right) \hat{x} + \mathcal{O}\left( \frac{1}{r}\right)\right\}, \\
			u^{s}_{s}\left( x\right) &= \frac{\text{e}^{ik_sr}}{\sqrt{r}} \left\{ u^{\infty}_s \left( \hat{x} \right) \hat{x}^{\perp} + \mathcal{O}\left( \frac{1}{r}\right)\right\}
		\end{split}
	\end{equation}
	as $r = |x| \to \infty$, where $u^{\infty}_p$ and $u^{\infty}_s$ are both scalar functions defined on $\mathbb{S}$. Hence, a Kupradze radiating solution has the asymptotic behavior 
	\begin{equation}
		u^{s}\left( x \right) = \frac{\text{e}^{ik_pr}}{\sqrt{r}} u^{\infty}_p \left( \hat{x} \right) \hat{x} + \frac{\text{e}^{ik_sr}}{\sqrt{r}} u^{\infty}_s \left( \hat{x} \right) \hat{x}^{\perp } + \mathcal{O}\left( \frac{1}{r^{\frac{3}{2}}}\right) \quad \text{as} \quad r \to \infty.
	\end{equation}
	The far-field pattern $u^{\infty}$ of $u^{s}$ is defined as 
	\begin{equation}
		u^{\infty} \left( \hat{x}\right) \coloneqq u^{\infty}_p \left( \hat{x}\right) \hat{x} +
		u^{\infty}_s \left( \hat{x}\right) \hat{x}^{\perp }.
	\end{equation}
	Then the compressional and shear parts of the far field are uniquely determined by $u^{\infty}$ as 
	\begin{equation}
		u^{\infty}_p \left( \hat{x}\right) = u^{\infty} \left( \hat{x}\right) \cdot \hat{x}, \quad u^{\infty}_s \left( \hat{x}\right) = u^{\infty} \left( \hat{x}\right) \cdot \hat{x}^{\perp}.
	\end{equation}
Introduce the compressional and shear parts of $u^{i}$ by
	\begin{equation}
		u^{i}_p\left(x; d\right)  \coloneqq u^{i}\left(x; d, 1,0 \right), \  u^{i}_s\left(x; d\right)  \coloneqq u^{i}\left(x; d, 0, 1 \right).
	\end{equation}
	Obviously, 
	\begin{eqnarray}
		u^{i}(x; d, c_p, c_s) = c_p u^{i}_p(x; d) + c_s u^{i}_s(x; d).
	\end{eqnarray}
	Correspondingly, for the far-filed pattern, we have
	\begin{equation}\label{u-inftypp}
		\begin{split}
			u^{\infty}_{pp} \left( \hat{x};d\right) & \coloneqq u^{\infty}_p \left( \hat{x}; d, 1,0\right), \\
			u^{\infty}_{sp} \left( \hat{x};d\right) & \coloneqq u^{\infty}_s \left( \hat{x}; d, 1,0\right),\\
			u^{\infty}_{ps} \left( \hat{x};d\right) & \coloneqq u^{\infty}_p \left( \hat{x}; d, 0,1\right),\\
			u^{\infty}_{ss} \left( \hat{x};d\right) & \coloneqq u^{\infty}_s \left( \hat{x};d,  0,1\right).
		\end{split}
	\end{equation}
	Thus, we obtain
	\begin{equation}\label{u-infty}
		u^{\infty} (\hat{x}; d, c_p, c_s) = \left(c_p u^{\infty}_{pp}(\hat{x}) + c_s u^{\infty}_{ps}(\hat{x})\right) \hat{x} + \left(c_p u^{\infty}_{sp}(\hat{x}) + c_s u^{\infty}_{ss}(\hat{x})\right)
		\hat{x}^{\perp},
	\end{equation}
where the dependence of $u^\infty_{\alpha\beta}$ $(\alpha, \beta=p,s)$ on $d$ has been omitted for simplicity.
	In this paper, the following inverse elastic scattering problems will be considered:

	\textbf{IP-P:} Reconstruct the shape and position of the scattere $D$ from  knowledge of the compressional part $u^{\infty}_{pp}(\hat{x})$ of the far-field pattern due to one incident compressional wave $u^{i}_p$.
	
	\textbf{IP-S:} Reconstruct the shape and position of  $D$ from knowledge of the shear part $u^{\infty}_{ss}(\hat{x})$ of the far-field pattern due to one incident shear wave $u^{i}_s$.

	\textbf{IP-F:} Reconstruct the shape and position of  $D$ from using the entire far-field pattern $u^{\infty}(\hat{x})$ due to one incident wave $u^{i}$.

	\section{Factorization method with infinitely many plane waves}\label{sec:3}
	
	\subsection{Review of the classical Factorization method for inverse elastic scattering}

	Given a vector field $g(d)=g_p(d) d + g_s(d) d^{\bot} \in \left(L^2(\mathbb{S})\right)^2$, the superposition of plane waves
	\begin{equation}
		v_g(x):= \text{ e}^{-i\frac{\pi}{4}} \int_{\s} \left\{\sqrt{\frac{k_p}{\omega}}d\text{e}^{ik_px \cdot d}g_p(d) + \sqrt{\frac{k_s}{\omega}}d^{\bot}\text{e}^{ik_sx \cdot d}g_s(d) \right\} ds(d)
	\end{equation}
	is denoted as the elastic Herglotz wave function with density $g$. 
	The Green's tensor of the Navier equation in free space, also called Kupradze's tensor (see e.g., \cite{Arens}), is denoted by
	\begin{equation}
		\Gamma (x,y) \coloneqq \frac{i}{4\mu} H^{(1)}_0(k_s |x-y|) {\bf I} + \frac{i}{4\omega^2}\triangledown^{\bot }_x \triangledown_x\left(H^{(1)}_0(k_s|x-y|) - H^{(1)}_0(k_p|x-y|)\right), \ x,y\in \mathbb{R}^2, \ x \neq y,
	\end{equation}
	where $H^{(1)}_0$ is the Hankel function of the first kind and of order $n$. For any $y \in \R^2$ and any direction $a \in \s$, an elastic point source in $y$ with the polarization $a$ is given by 
	\begin{equation}
		u(x) = \Gamma(x,y)a, \quad x \in \R^2\backslash\{y\}.
	\end{equation}
	The far-field pattern $\Gamma^{\infty}(\cdot,y;a)$ of this point source is given by 
	\begin{equation}
		\Gamma^{\infty}(\hat{x},y;a) = \frac{k^2_p}{\omega^2} \frac{\text{e}^{i\frac{\pi}{4}}}{\sqrt{8\pi k_p}} \text{e}^{-ik_p\hat{x}\cdot y} (\hat{x}\cdot a) \hat{x} + \frac{k^2_s}{\omega^2} \frac{\text{e}^{i\frac{\pi}{4}}}{\sqrt{8\pi k_s}} \text{e}^{-ik_s\hat{x}\cdot y} (\hat{x}^{\bot}\cdot a) \hat{x}^{\bot},
	\end{equation}
with the compressional and shear parts:
	\begin{equation}
	    \Gamma^{\infty}_p(\hat{x},y;a)= \Gamma^{\infty}(\hat{x},y;a) \cdot \hat{x}, \quad \Gamma^{\infty}_s(\hat{x},y;a) = \Gamma^{\infty}(\hat{x},y;a) \cdot \hat{x}^{\bot}.
	\end{equation}
In this paper we define the elastic far-field operator as follows.
	\begin{definition}
		The far-field operator $F_D : \left(L^2(\mathbb{S})\right)^2 \to \left(L^2(\mathbb{S})\right)^2$ is defined by 
		\begin{equation}
			\begin{split}
				(F_D g)(\hat{x})  & \coloneqq \text{ e}^{-i\frac{\pi}{4}} \int_{\mathtt{S}} u^{\infty}_D \left(\hat{x}; d, \sqrt{\frac{k_p}{\omega}}g_p(d), \sqrt{\frac{k_s}{\omega}}g_s(d)\right) ds(d) \\ 
				&= \text{ e}^{-i\frac{\pi}{4}} \int_{\mathtt{S}} \left\{\sqrt{\frac{k_p}{\omega}}u^{\infty}_D(\hat{x}; d, 1, 0)g_p(d) + \sqrt{\frac{k_s}{\omega}}u^{\infty}_D(\hat{x}; d, 0, 1)g_s(d)\right\} ds(d).
			\end{split}
		\end{equation}
	\end{definition}
For rigid elastic bodies, it is well known that $F_D$ is a normal operator. It was proved in \cite[Theorem 4.3]{Arens} that the operator $F_D$ can be decomposed into the form
	\begin{equation}\label{fd-svd}
		F_D = -\sqrt{8\pi \omega} G_D S^{\ast}_D G^{\ast}_D.
	\end{equation}
	Here the data-to-patten operator $G_D : \left(H^{1/2}(\partial D)\right)^2 \to \left(L^2(\mathbb{S})\right)^2$ is defined by $G_D h \coloneqq u^{\infty}$, where $u^{\infty}$ is the far-field pattern of the solution to the Dirichlet boundary value problem of the Navier equation with the boundary value $h$. 
	The operator $S^{\ast}_D:\left(H^{-1/2}(\partial D)\right)^2 \to \left(H^{1/2}(\partial D)\right)^2$ is the adjoint of the elastic single layer potential operator $S_D$, given by 
	\begin{equation}
		S_D \phi (x) \coloneqq \int_{\partial D} \Gamma (x,y) \phi(y) ds(y), \quad x \in \partial D. 
	\end{equation}
	By the Factorization method, the far-field pattern $\Gamma^{\infty}(\cdot,z;a)$ belongs to the range of $G_D$ if and only if $z \in D$ (see \cite[Theorem 4.7]{Arens}). Moreover, the $(F^{\ast}F)^{1/4}$-method (see \cite[Theorem 4.8]{Arens}) verifies the relation $\text{Range} (G_D)= \text{Range} ((F^{\ast}_D F_D)^{1/4})$ if $\omega^2$ is not an eigenvalue of $-\Delta^{\ast}$ over $D$. Hence, by the Picard theorem, the scatterer $D$ can be characterized by the spectra of $F_D$ as follows.
	\begin{theorem}\label{classfm}
		(\cite[Theorem 4.8]{Arens}) Assume that $\omega^2$ is not a Dirichlet eigenvalue of $-\Delta^{\ast}$ over $D$. Denote by $(\lambda^{(n)}_{D}, \varphi^{(n)}_{D})$ a spectrum system of the 
		far-field operator $F_D : \left(L^{2}(\mathbb{S})\right)^2 \rightarrow \left(L^{2}(\mathbb{S})\right)^2$.
		Then,
		  \begin{equation}\label{FD}
			z \in D  
			\Longleftrightarrow
		  	I(z) \coloneqq \sum\limits_{n \in \mathbb{Z}} \frac{\left|\left\langle \Gamma^{\infty}(\cdot,z;a), \varphi_{D}^{(n)}\right\rangle_{\mathbb{S}}\right|^{2}}{\left|\lambda_{D}^{(n)}\right|} < +\infty.
		  \end{equation}
	\end{theorem}
By Theorem \ref{classfm}, the sign of the indicator function $I(z)$ can be regarded as the characteristic function of $D$. We note that in \eqref{FD}, $z\in \R^2$ are the sampling variables/points and the spectral data $(\lambda^{(n)}_{D}, \varphi^{(n)}_{D})$ are determined by the far-field patterns $u_D^\infty(\hat{x},d)$ over all observation and incident directions $\hat{x}, d\in \s$.

Below we state the Factorization method which involves only the compressional or shear plane waves.
Introduce the projection space $L^2_p(\s) \coloneqq \left\{g_p : \s \to \C, g_p(d) = g(d) \cdot d, |g_p| \in L^2(\s)\right\}$ and $L^2_s(\s) \coloneqq \left\{g_s : \s \to \C, g_s(d) = g(d) \cdot d^{\bot}, |g_s| \in L^2(\s)\right\}$.  Define the projection operators $P_p : \left(L^2(\s)\right)^2 \to L^2_p(\s)$ and $P_s : \left(L^2(\s)\right)^2 \to L^2_s(\s)$ (see e.g., \cite{HuKirsch2012}) by
	\begin{equation}
		P_pg(d) := g_p(d), \quad P_sg(d):=g_s(d).
	\end{equation}
	Thus, we can define the P-part and S-part of the far-field operator $F_D$. 
	\begin{definition}\label{def-ps}
		The far-field operator $F^{(p)}_D : L^2_p(\s) \to L^2_p(\s)$ is defined by 
		\begin{equation}
			F^{(p)}_D g_p(d) \coloneqq P_p F_D P^{\ast}_p g_p(d).	
		\end{equation}
	Analogously,	the far-field operator $F^{(s)}_D : L^2_s(\s) \to L^2_s(\s)$ is defined by
		\begin{equation}
			F^{(s)}_D g_s(d) \coloneqq P_s F_D P^{\ast}_s g_s(d).	
		\end{equation}
Here	$P^{\ast}_p$ and $P^{\ast}_s$ are the adjoint operators of $P_p$ and $P_s$, respectively.
	\end{definition}
 By \eqref{fd-svd}, we have the factorization 
	\begin{equation}\label{3.13}
		F^{(\alpha)}_D = -\sqrt{8\pi \omega}(P_{\alpha}G_D)S^{\ast}_D(P_{\alpha}G_D)^{\ast}, 
	\end{equation}
	where $\alpha = p, s$.
Based on (\ref{3.13}), the $F_{\#}$-method (see \cite[Lemma 3.5]{HuKirsch2012}) verifies the relation $\text{Range} (P_{\alpha}G_D)= \text{Range} ((F^{(\alpha)}_{D,\#} )^{1/2})$, provided $\omega^2$ is not a Dirichlet eigenvalue of $-\Delta^{\ast}$ over $D$. Here the operator $F_{\#}$ is defined by 
	\begin{equation}
		F_{\#} \coloneqq |\text{Re}F| + |\text{Im}F|, \quad \text{Re}F \coloneqq \frac{1}{2}[F+F^{\ast}], \quad \text{Im}F \coloneqq \frac{1}{2i}[F-F^{\ast}].
	\end{equation}
	Hence, the scatterer $D$ can be characterized by the spectra of $F^{(\alpha)}_D$ as follows.
	\begin{theorem}
		(\cite[Theorem 3.7, 3.8]{HuKirsch2012}) Assume that $\omega^2$ is not a Dirichlet eigenvalue of $-\Delta^{\ast}$ over $D$. Denote by $(\lambda^{(n)}_{D,\alpha}, \varphi^{(n)}_{D,\alpha})$ a spectrum system of the positive operator $F^{(\alpha)}_{D,\#}$.
		Then,
		  \begin{equation}
			z \in D  
			\Longleftrightarrow
		  	I^{(\alpha)}(z) \coloneqq \sum\limits_{n \in \mathbb{Z}} \frac{\left|\left\langle \Gamma^{\infty}_{\alpha}(\cdot,z;a), \varphi_{D,\alpha}^{(n)}\right\rangle_{\mathbb{S}}\right|^{2}}{\left|\lambda_{D,\alpha}^{(n)}\right|} < +\infty, \quad \alpha=p, s.
		  \end{equation}
	\end{theorem}

	\subsection{Further discussions on Factorization method}
Before stating the one-wave version of the factorization method for inverse elastic scattering, we first present a corollary of Theorems \ref{classfm}. Denote by $\Omega\subset \R^2$ a convex and bounded Lipschitz domain which represents a rigid elastic scatterer.  Here we use a new notation $\Omega$ in order to distinguish from our target scatterer $D$. The far-field operator $F_{\Omega} : \left(L^{2}(\mathbb{S})\right)^2 \rightarrow \left(L^{2}(\mathbb{S})\right)^2$ corresponding to $\Omega$ is therefore defined by
	\begin{equation}\label{op-domain}
		(F_{\Omega} g)(\hat{x}) \coloneqq \text{ e}^{-i\frac{\pi}{4}} \int_{\mathbb{S}} \left\{\sqrt{\frac{k_p}{\omega}}u^{\infty}_{\Omega}(\hat{x}; d, 1, 0)g_p(d) + \sqrt{\frac{k_s}{\omega}}u^{\infty}_{\Omega}(\hat{x}; d, 0, 1)g_s(d)\right\} ds(d),
	\end{equation}
	where $u^{\infty}_{\Omega} (\hat{x}; d, 1, 0)$ and $u^{\infty}_{\Omega} (\hat{x}; d, 0, 1)$ are the far-field patterns corresponding to the elastic plane waves $u^{i}_p(x; d)$ and $u^{i}_s(x; d)$ incident onto $\Omega$, respectively. The eigenvalues and eigenfunctions of $F_{\Omega}$ will be denoted by $(\lambda^{(n)}_{\Omega}, \varphi^{(n)}_{\Omega})$.
The following Corollaries can be derived straightforwardly from the classical Factorization method in the previous subsection.

	\begin{corollary}\label{fm_convex_domain} Let $v^\infty\in \left(L^2(\s)\right)^2$ and assume that
		$\omega^2$ is not a Dirichlet eigenvalue of $-\Delta^{\ast}$ over $\Omega$. Then 
	   \begin{equation}
		   I(\Omega) = \sum_{n \in \mathbb{Z}} \frac{\left|\left\langle v^{\infty}, \varphi^{(n)}_{\Omega}\right\rangle_{\mathbb{S}}\right|^{2}}{\left|\lambda^{(n)}_{\Omega}\right|} < + \infty
	   \end{equation}
		   if and only if $v^{\infty}$ is the far-filed pattern of some Kupradze radiating solution $v^{s}$, where $v^{s}$ satisfies the Navier equation
	   \begin{equation}\label{Navier}
		   \Delta^{\ast} v^{s} + \omega^{2} v^{s} = 0 \qquad \text{in} \quad  \mathbb{R}^2 \backslash \overline{\Omega},
	   \end{equation}
		with the boundary data $v^{s} |_{\partial \Omega}\in \left(H^{1/2}(\partial \Omega)\right)^2.$
	   \end{corollary}

	   \begin{proof}
		By \eqref{fd-svd}, we have
		$F_{\Omega} = -\sqrt{8\pi \omega} G_{\Omega} S^*_{\Omega} G^*_{\Omega}$,
		where $G_{\Omega}: \left(H^{1/2}(\partial \Omega)\right)^2 \to \left(L^2(\mathbb{S})\right)^2$ is the data-to-pattern operator corresponding to $\Omega$.
		Obviously, $I(\Omega) <+\infty$ if and only if $v^{\infty} \in \text{Range}((F^*_{\Omega} F_{\Omega})^{1/4})$. Since $\text{Range}((F^*_{\Omega} F_{\Omega})^{1/4}) = \text{Range}(G_{\Omega})$, we get $v^{\infty} \in \text{Range}(G_{\Omega})$ if and only if $I(\Omega) <+\infty$.
		Recalling the definition of $G_{\Omega}$, it follows that $v^{s}$ satisfies the Navier equation \eqref{Navier} and the Kupradze radiation condition \eqref{Kupr-cd} with the boundary data $v^{s} |_{\partial \Omega}\in \left(H^{1/2}(\partial \Omega)\right)^2.$
		\end{proof}

		\begin{corollary}\label{fm_convex_domainpp} 
			Let $w^\infty_{\alpha\alpha} \in L^2_{\alpha}(\s)$($\alpha = p, s$) and assume that
			$\omega^2$ is not a Dirichlet eigenvalue of $-\Delta^{\ast}$ over $\Omega$. Denote by $(\lambda^{(n)}_{\Omega,\alpha}, \varphi^{(n)}_{\Omega,\alpha})$ a spectrum system of the positive operator $F^{(\alpha)}_{\Omega,\#}$. Then 
		   \begin{equation}\label{OP}
			   I^{(\alpha)}(\Omega) = \sum_{n \in \mathbb{Z}} \frac{\left|\left\langle w^\infty_{\alpha\alpha}, \varphi^{(n)}_{\Omega, \alpha}\right\rangle_{\mathbb{S}}\right|^{2}}{\left|\lambda^{(n)}_{\Omega, \alpha}\right|} < + \infty
		   \end{equation}
			   if and only if $w^\infty_{\alpha\alpha}= P_{\alpha} v^{\infty}$, where $v^{\infty}$ is the far-field pattern of some Kupradze radiating solution $v^{s}$, which is defined in $\mathbb{R}^2 \backslash \overline{\Omega}$ and $v^{s} |_{\partial \Omega}\in \left(H^{1/2}(\partial \Omega)\right)^2$. 	
That is,	$w^\infty_{\alpha\alpha}$ is the far-field pattern of some Sommerfeld radiating solution $w^s_{\alpha\alpha}$, fulfilling the relations
 \begin{equation}\label{wsp}
 w_{pp}^{s} = -\frac{1}{k_p^2}\divv v^s\quad \mbox{if}\quad\alpha = p;\quad w_{ss}^{s} = \frac{1}{k_s^2}\curl v^s\quad\mbox{if} \quad \alpha = s, 
 \end{equation}
 where $v^s$ satisfies the boundary value problem of the Navier equation 
		   \begin{equation}\label{navierbp}
			   \Delta^{\ast} v^{s} + \omega^{2} v^{s} = 0 \ \text{in} \  \mathbb{R}^2 \backslash \overline{\Omega}, \quad v^{s} |_{\partial \Omega}\in \left(H^{1/2}(\partial \Omega)\right)^2.
		   \end{equation}
		\end{corollary}
Similar to Corollary \ref{fm_convex_domain},  Corollary \ref{fm_convex_domainpp} can be proved by using the projection operator $P_{\alpha}$ and the $F_{\#}$-method. 
		\begin{proof}
		Without loss of generality, we assume that the relation (\ref{OP}) holds with $\alpha=p$. By the Picard theorem,   $w^\infty_{pp} \in  \text{Range}(F^{(p)}_{\Omega, \#}) = \text{Range}(P_pG_{\Omega})$ if and only if the indicator function $I^{(p)}(\Omega)  < + \infty$. This means that $w^\infty_{pp}=P_p v^{\infty}$, where $v^{\infty}$ is the far-field pattern of a Kupradze's radiating solution $v^s$ to the boundary value problem \eqref{navierbp}. Using the Hodge decomposition, we see $w^\infty_{pp}$ is the far-field pattern of $w^s_{pp}$, where $w^s_{pp}=-\frac{1}{k_p^2}\text{div} v^s$. 

		On the other hand, let $v^s$ be a solution to the boundary value problem \eqref{navierbp} and define  $w_{pp}^{s} := -\frac{1}{k_p^2}\divv v^s$. Suppose that 
 $w^\infty_{pp}$ is the far-field pattern of $w^s_{pp}$. Then, by the Hodge decomposition  it follows that $w^\infty_{pp} = P_p v^{\infty}=P_p G_\Omega(v^s|_{\partial \Omega})$, implying that $w^\infty_{pp} \in \text{Range}(P_pG_{\Omega})=\text{Range}(F^{(p)}_{\Omega, \#}) $. Applying the Picard theorem yields $I^{(p)}(\Omega) < + \infty$.		
		\end{proof}
\begin{remark}
It follows from (\ref{wsp}) that the restrictions of $w^s_{pp}$, $w^s_{ss}$ to $\partial \Omega$ lie in the space $H^{-1/2}(\partial\Omega)$.
\end{remark}

	\section{Factorization method with one plane waves}\label{sec:4}


		In our applications of Corollary \ref{fm_convex_domain} (resp. Corollary \ref{fm_convex_domainpp}), we will take $v^\infty$ to be the measurement data $u_D^\infty(\hat x; d_0, c_p, c_s)$ (resp. $u_{D,\alpha\alpha}^\infty(\hat x; d_0)$, $\alpha=p,s$)	corresponding to our target elastic scatterer $D$ and the incident elastic plane wave $u^{i}(x; d_0, c_p, c_s)$ (resp. $u^{i}_\alpha(\hat{x};d_0)$) for some fixed $d_0\in\s$. We shall omit the dependance on $d_0$ if it is always clear from the context.  Our purpose is to extract the geometrical information on $D$ from the domain-defined indicator functions $I(\Omega)$ and $I^{(\alpha)}(\Omega)$ ($\alpha=p,s$). By Corollary \ref{fm_convex_domain}, $I(\Omega)<\infty$ if the scattered field $u^{s}_D(x)=u^{s}(x; d_0, D)$ can be extended to the domain $\R^2\backslash\overline{\Omega}$ as a solution to the Navier equation. 
Below we shall discuss the absence of the analytic extension of $u^{s}_D$, $u^{s}_{D, p}$ and $u^{s}_{D,s}$ around a planer corner point of $D$.

\begin{lemma}\label{ana-ext}
			Assume that $D$ is a rigid elastic scatterer occupying a convex polygon. 
Then the scattered field $u^{s}_D(x; d_0)$, $u^{s}_{D, p}(x; d_0)$ and $u^{s}_{D,s}(x; d_0)$
 cannot be analytically extended from $\mathbb{R}^{2} \backslash \overline{D}$ into $D$ across any corner of $D$.
		\end{lemma}


		\begin{proof} We shall carry out the proof by contradiction.
		 
	(i)	Assume on the contrary that $u^{s}_D(x;d_0)$ can be analytically continued across a corner of $\partial D$. By coordinate translation and rotation, we may suppose that this corner coincides with the origin, so that $u^{s}_D(x;d_0)$ and also the total field $u_D=u^{s}_D(x;d_0)+u^{i}(x; d_0, c_p, c_s)$ satisfy the Navier equation \eqref{lame-eqn} in $B_\epsilon(O)$ for some $\epsilon>0$. Since $u_D$ is real analytic in $(\R^2\backslash\overline{D})\cup B_\epsilon(O)$ and $D$ is a convex polygon, $u_D$ satisfies the Navier equation on the closure of an infinite sector $\Sigma\subset \R^2\backslash\overline{D}$ which extends the finite sector $B_\epsilon(O)\cap D$ to $\R^2\backslash\overline{D}$. In particular, the total field $u_D$ fulfills the Dirichlet boundary condition on the two half lines $\partial \Sigma$ starting from the corner point $O$. 
Since $u^{s}_D(x;d_0)$ fulfills the Kupradze's  radiation condition, it holds that
\[
\lim_{|x|\rightarrow\infty} u^{i}(x; d_0, c_p, c_s)=-
\lim_{|x|\rightarrow\infty} u^{s}_D(x; d_0, c_p, c_s)=0,\quad x\in \partial \Sigma.
\]
However, this is impossible for an elastic plane wave incidence of the form \eqref{2.1}.		

(ii) Suppose that $u^{s}_{D,p}$ can be analytically extended from $\R \backslash \bar{D}$ into $D$ across a corner $O$ of $D$. 
That is,  $u^{s}_{D,p}$ extends to a function which satisfies the vector Helmholtz equation $\Delta u^{s}_{D,p} + k_p^2 u^{s}_{D,p} = 0$ in $B_\epsilon(O)$ for some $\epsilon>0$.
By the Hodge decomposition of the total field $u_D = u_{D,p} + u_{D,s} = \triangledown \varphi + \curl \psi$,  the function $\varphi$ can be also extended to $B_\epsilon(O)$ as a solution to the scalar Helmholtz equation. Here we have used the fact that the incident wave $u^{i}$ is an entire solution to the Navier equation. 
In particular, this implies that the normal and tangential derivatives $(\partial_n \varphi, \partial_\tau \varphi)$ of $\varphi$ are both piecewise analytic on $\partial D\cap B_\epsilon(O)$. Recalling the Dirichlet boundary condition of $u_D$, we have
\begin{equation}
			\left\{
				\begin{split}
					u_D\cdot n &= \frac{\partial \varphi}{\partial n} + \frac{\partial \psi}{\partial \tau} = 0 \\
					u_D\cdot \tau &= \frac{\partial \varphi}{\partial \tau} - \frac{\partial \psi}{\partial n} = 0
				\end{split}
			\right. \quad \text{on} \ \partial D\cap B_\epsilon(O).
		\end{equation}
Hence, $(\partial_\tau\psi, \partial_n \psi)=(-\partial_n\varphi, \partial_\tau\varphi)$ on $\partial D\cap B_\epsilon(O)$ and thus the Cauchy data of $\psi$ are also piecewise analytic. By the Cauchy-Kovalevskaya theorem,  the function $\psi$ admits an extension from $B_\epsilon(O)\cap (\R^2\backslash\overline{D})$ to $B_\epsilon(O)\cap D$, as a solution to the scalar Helmholtz equation with the wave number $k_s$.
Hence, the total field can be continued to $B_\epsilon(O)\cap D$, which however is impossible by the first part of the proof. 

(iii) The case of $u^{s}_{D,s}$ can be proved similarly to the second step for 
$u^{s}_{D,p}$.
		\end{proof}

As in the acoustic case \cite{M-H},  our approach applies to inverse source problems as well. For this purpose, we need to justify the absence of analytical extension for elastic source scattering problems in a corner domain, which is closet to studies of non-radiating elastic sources given in \cite{Bla2018}.
		\begin{lemma}
Let $\chi_D$ be the characteristic function for the convex polygon $D$. If $u\in \left(H^2_{loc}(\mathbb{R}^2)\right)^2$ is a radiating solution to
			\begin{equation}
			\Delta^{\ast} u(x)+ \omega^2u(x)=\chi_D(x)f(x)\quad\mbox{in}\quad \mathbb{R}^2,
			\end{equation}
			where $f\in L^\infty(\R^2)$ is H\"older continuous near the corner point $O$ of $D$ satisfying $f(O)\neq 0$. Then $u$ cannot be analytically extended from $\mathbb{R}^2\backslash\overline{D}$ to $D$ across the  corner $O$. 
		\end{lemma}

		\begin{figure}[ht]
			\centering
			\includegraphics[scale=0.3]{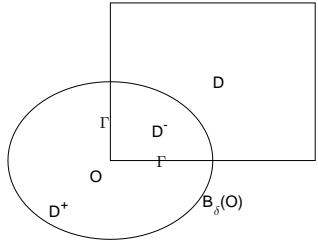}  
			\caption{Illustration of a convex polygonal source term where $O$ is corner point of $D$.}\label{pic2}
		\end{figure}

		\begin{proof}Without loss of generality, the corner point $O$ is supposed to coincide with the origin. Set $v^\pm=u|_{D^\pm}$ where $D^+:=B_\delta(O)\cap (\R^2\backslash\overline{D})$ and $D^-:=B_\delta(O)\cap D$. 
		Assume that $v^+$ can be analytically extended from $\R^2\backslash\overline{D}$ to $B_\delta(O)$ for some $\delta > 0$ (see Figure \ref{pic2}), as a solution of the Navier equation. This implies that
$\Delta^{\ast} v^+ + \omega^2v^+ = 0$ in $B_\delta(O)$,
with the Cauchy data 
\[
v^-=v^+, \quad \mathcal{T}_{\nu} v^-=\mathcal{T}_{\nu} v^+\quad \mbox{on}\quad\Gamma:=\partial D\cap B_\delta(O).
\] Here, the boundary traction operator $\mathcal{T}_{\nu}$ is defined as 
\[
\mathcal{T}_{\nu}u = 2\mu \frac{\partial u}{\partial \nu} + \lambda \nu \triangledown \cdot u + \mu \nu^{\bot}(\partial_2 u_1 - \partial_1 u_2)
\] in the two-dimensional case. 
Since $\Delta^{\ast} v^- + \omega^2v^- = f$ in $B_\delta(O)\cap D$, the difference $w:=u^--u^+$ is a solution to 
			\begin{equation}
					\Delta^{\ast} w + \omega^2w = f\quad\mbox{in}\quad B_\delta(O)\cap D,\quad
					w=\mathcal{T}_{\nu}w=0\quad\mbox{on}\quad \Gamma.
			\end{equation}
By \cite[Proposition 3.2]{Bla2018}, it follows that $f(O)=0$, which is in contradiction with our assumption  that $f(O)\neq 0$.
		\end{proof}


To state the one-wave factorization method, we shall restrict our discussions to a convex polygonal rigid elastic scatterer $D$.  
Let $\Omega$ be another convex rigid scatterer for detecting $D$ such that $\omega^2$ is not the Dirichlet eigenvalue of $-\Delta^{\ast}$ in $\Omega$. Denote by $(\lambda^{(n)}_{\Omega}, \varphi^{(n)}_{\Omega})$ the eigenvalues and eigenfunctions of the far-field operator $F_{\Omega}$. Below we characterize the inclusion relationship  between our target scatterer $D$ and the test domain $\Omega$ by the measurement data $u_D^\infty$ and the spectra of $F_{\Omega}$.
		\begin{theorem}\label{One-FM}
			Define	
			\begin{equation}\label{svd-omega}
				W(\Omega) := \sum_{n \in \mathbb{Z}} \frac{\left|\left\langle u_D^{\infty}, \varphi^{(n)}_{\Omega}\right\rangle_{\mathbb{S}}\right|^{2}}{\left|\lambda^{(n)}_{\Omega}\right|}.
			\end{equation}
			Then $W(\Omega)<\infty$	if and only if $D \subseteq \Omega$.
		\end{theorem}

		\begin{proof}
		By Corollary \ref{fm_convex_domain}, $W(\Omega) < +\infty$ implies that $u^{s}_D(x;d_0)$ is analytic in $\mathbb{R}^2 \backslash \overline{\Omega}$.
		If  $D \nsubseteq \Omega$, three cases might happen: (i) $\Omega\subset D$; (ii) $\Omega\cap D=\emptyset$; (iii) $\Omega\cap D\neq \emptyset$ and $\Omega\cap (\R^2\backslash\overline{D})\neq \emptyset$.
		In either of these cases, we observe that there is always a corner $O \in D$ and $O \notin \Omega$ by the convexity of both $\Omega$ and $D$. Then, $u^{s}_D(x;d_0)$ can be analytically continued from $\R^2\backslash\overline{D}$ to $D$ across the corner $O$ of $\partial D$, which however is impossible by  Lemma \ref{ana-ext}. This proves the relationship $D \subseteq \Omega$.

	Now assume that $D \subseteq \Omega$. Then the scattered field $u^{s}_D(x;d_0)$ satisfies the Navier equation $ \Delta^{\ast} u^{s}_D(x;d_0) + \omega^{2} u^{s}_D(x;d_0) = 0$ in  $\mathbb{R}^2 \backslash \overline{\Omega}$ with the boundary data  $f:=u^{s}_D(x;d_0)|_{\partial \Omega} \in \left(H^{1/2}(\partial \Omega)\right)^2$. 
	This implies that $u_D^\infty=G_\Omega(f)$. Hence,
we get $W(\Omega) < +\infty$ by applying Corollary \ref{fm_convex_domain}.
		\end{proof}

		For the operator $F^{(\alpha)}_{\Omega}$($\alpha=p, s$), denote by $(\lambda^{(n)}_{\Omega,\alpha}, \varphi^{(n)}_{\Omega,\alpha})$ the eigenvalues and eigenfunctions of the positive operator $F^{(\alpha)}_{\Omega,\#}$. 
Using Corollary \ref{fm_convex_domainpp} and arguing analogously to the proof of Theorem \ref{One-FM}, we obtain
	\begin{theorem}\label{One-FMpp}
			Define	
			\begin{equation}
				W^{(\alpha)}(\Omega) := \sum_{n \in \mathbb{Z}} \frac{\left|\left\langle u_{D,\alpha \alpha}^{\infty}, \varphi^{(n)}_{\Omega,\alpha}\right\rangle_{\mathbb{S}}\right|^{2}}{\left|\lambda^{(n)}_{\Omega,\alpha}\right|},
			\end{equation}
			where $u_{D,\alpha \alpha}^{\infty}$ is defined the same as \eqref{u-inftypp} for the scatterer $D$, $\alpha = p, s$. Then $W^{(\alpha)}(\Omega)<\infty$	if and only if $D \subseteq \Omega$.
		\end{theorem}

	\section{Explicit examples when $\Omega$ is a disk}\label{sec:5}
	
	Theorems \ref{One-FM} and \ref{One-FMpp} rely essentially on the factorization form (see e.g., \eqref{fd-svd}) of the far-field operator for the elastic scatterer $\Omega$.  Below we show that the results of Theorems \ref{One-FM} and \ref{One-FMpp} can be justified independently of the factorization form, as long as the test domain $\Omega$ is chosen to be a rigid elastic disk. This is mainly due to the explicit form of  the far-field pattern for a rigid disk in terms of special functions; see Subsection \ref{sec5.1} below. Then we can get an explicit spectral system of the far-field operator $F_{B_R}$ in Subsection \ref{sec5.2}. The proofs 
 Corollaries \ref{fm_convex_domain} and  \ref{fm_convex_domainpp}  will be shown in Subsection \ref{sec5.3}.	
Note that this section is of independent interests, since as shown in the subsequent sections, the derivation of eigenvalues and eigenfunctions of $F_{B_R}$ turns out to be non-trivial, which is in contrast to the acoustic case of the Helmholtz equation. 
	
	\subsection{Far-field pattern of a rigid disk $B_R$}\label{sec5.1}
Assume that $B_R \coloneqq \left\{x : |x|<R \right\}$ is a rigid disk centered at the origin with radius $R>0$. Let $\hat{x}=\left(\cos \theta_x, \sin \theta_x \right)^T\in \s$ be the observation direction (or variable) of the elastic far-field pattern.
According to the Hodge decomposition \eqref{Hodgedecom}, we can introduce scalar functions $\varphi\left(r,\theta_x\right)$ and $\psi\left(r,\theta_x\right)$ such that in polar coordinates $x=(r, \theta_x)$,
	\begin{equation}\label{5.1}
		u^{s}(r,\theta_x) = \text{grad}\  \varphi (r, \theta_x) + \text{curl}\  \psi (r, \theta_x). 
	\end{equation}
Recall the relationship between the Cartesian and polar coordinates for gradient: 
	\begin{equation}
		\begin{pmatrix}
			\partial_1 \\ \partial_2
		\end{pmatrix}
		= 
		\begin{pmatrix}
			\cos \theta & -\frac{1}{r}\sin \theta \\
			\sin \theta & \frac{1}{r} \cos \theta
		\end{pmatrix}
		\begin{pmatrix}
			\partial_r \\ \partial_{\theta}
		\end{pmatrix}.
	\end{equation}
Noting that $\varphi$ and $\psi$ are both Sommerfeld radiating solutions, we make the following ansatz: 
 \[
 \varphi = \sum_{n \in \mathbb{Z}} \frac{1}{\sqrt{k_p}} A_n H^{(1)}_n(k_p r)\text{e}^{in\theta_x},\quad
 \psi = \sum_{n \in \mathbb{Z}} \frac{1}{\sqrt{k_s}} B_n H^{(1)}_n(k_s r)\text{e}^{in\theta_x}.
 \] We can get  from (\ref{5.1}) that
	\begin{equation}
		\begin{split}
			u^{s}(r,\theta_x) &= 
			\begin{pmatrix}
				\cos \theta_x \partial_r \varphi - \frac{1}{r} \sin \theta_x \partial_{\theta_x} \varphi \\
				\sin \theta_x \partial_r \varphi + \frac{1}{r} \cos \theta_x \partial_{\theta_x} \varphi
			\end{pmatrix}
			+
			\begin{pmatrix}
				-\sin \theta_x \partial_r \psi - \frac{1}{r} \cos \theta_x \partial_{\theta_x} \psi \\
				\cos \theta_x \partial_r \psi - \frac{1}{r} \sin \theta_x \partial_{\theta_x} \psi
			\end{pmatrix} \\
			&= \sum_{n \in \mathbb{Z}} \frac{\text{e}^{in\theta_x}}{r} \left[
				\left(
					\sqrt{k_p} r H^{(1)\prime}_n(k_p r) A_n - \frac{in}{\sqrt{k_s}} H^{(1)}_n(k_s r)B_n
				\right) 
			\begin{pmatrix}
				\cos \theta_x \\ \sin \theta_x
			\end{pmatrix}
			\right. \\
			 & \quad \left.   + \left(
				\frac{in}{\sqrt{k_p}} H^{(1)}_n(k_p r)A_n + \sqrt{k_s} r H^{(1)\prime}_n(k_s r)B_n
				\right)
			\begin{pmatrix}
				-\sin \theta_x \\ \cos \theta_x
			\end{pmatrix}
			\right].
		\end{split}
	\end{equation}
	Using the asymptotic property of Hankel functions 
	\begin{equation}
		\begin{split}
			H_n^{(1)}(z) &= \sqrt{\frac{2}{\pi z}} \text{e}^{i(z-\frac{n\pi}{2}-\frac{\pi}{4})} \left( 1+\mathcal{O}(\frac{1}{z})\right), \ z \to \infty, \\
			H_n^{(1)\prime}(z) &= \sqrt{\frac{2}{\pi z}} \text{e}^{i(z-\frac{n\pi}{2}+\frac{\pi}{4})} \left( 1+\mathcal{O}(\frac{1}{z})\right), \ z \to \infty,
		\end{split}
	\end{equation}
	we have 
	\begin{equation}
		u^{s}(x) = \frac{\text{e}^{ik_pr}}{\sqrt{r}} \sqrt{\frac{2}{\pi}} \sum_{n \in \mathbb{Z}} \text{e}^{-i(\frac{n\pi}{2}-\frac{\pi}{4})} A_n \text{e}^{in\theta_x}\hat{x} + \frac{\text{e}^{ik_sr}}{\sqrt{r}} \sqrt{\frac{2}{\pi}} \sum_{n \in \mathbb{Z}} \text{e}^{-i(\frac{n\pi}{2}-\frac{\pi}{4})} B_n \text{e}^{in\theta_x}\hat{x}^{\bot} + \mathcal{O}(\frac{1}{r^{3/2}}).
	\end{equation}
	Thus, 
	\begin{equation}
		\begin{split}
			u^{\infty}_p(\hat{x}) = \sqrt{\frac{2}{\pi}} \text{e}^{i\frac{\pi}{4}} \sum_{n \in \mathbb{Z}} i^{-n} A_n \text{e}^{in\theta_x}, \quad
			u^{\infty}_s(\hat{x}) = \sqrt{\frac{2}{\pi}} \text{e}^{i\frac{\pi}{4}} \sum_{n \in \mathbb{Z}} i^{-n} B_n \text{e}^{in\theta_x}.
		\end{split}
	\end{equation}
	Now, we set 
	\begin{equation}
		t_{\alpha} = k_{\alpha} R, 
		\ \alpha = p, s,
	\end{equation}
	to get 
	\begin{equation}\label{uscR}
		\begin{split}
			u^{s}(r,\theta_x)|_{r=R} &= \sum_{n \in \mathbb{Z}} 
			\left[ 
				\left(
					\frac{t_p}{\sqrt{k_p}} H^{(1)\prime}_n(t_p) A_n - \frac{in}{\sqrt{k_s}} H^{(1)}_n(t_s)B_n
				\right)
				\begin{pmatrix}
					\cos \theta_x \\ \sin \theta_x
				\end{pmatrix} 
			\right. \\
			&\quad +\left. \left(
				\frac{in}{\sqrt{k_p}} H^{(1)}_n(t_p) A_n + \frac{t_s}{\sqrt{k_s}} H^{(1)\prime}_n(t_s)B_n
				\right)
				\begin{pmatrix}
					-\sin \theta_x \\ \cos \theta_x
				\end{pmatrix}
			\right]
			\frac{\text{e}^{in \theta_x}}{R} \\
			&=\sum_{n \in \mathbb{Z}} 
			\begin{pmatrix}
				\nu & \tau
			\end{pmatrix}
			\begin{pmatrix}
				t_p H^{(1)\prime}_n(t_p) & -in H^{(1)}_n(t_s)\\
				in H^{(1)}_n(t_p) & t_s H^{(1)\prime}_n(t_s)
			\end{pmatrix}
			\begin{pmatrix}
				\frac{1}{\sqrt{k_p}} & 0 \\
				0 & \frac{1}{\sqrt{k_s}}
			\end{pmatrix}
			\begin{pmatrix}
				A_n \\ B_n
			\end{pmatrix}
			\frac{\text{e}^{in \theta_x}}{R},
		\end{split}
	\end{equation}
	where $\nu = (\cos \theta_x, \sin \theta_x)^T$ and $\tau = (-\sin \theta_x, \cos \theta_x)^T$ are tangential and normal directions of $\partial B_R$ respectively, and that $(\nu\quad  \tau)$ is thus a 2-by-2 matrix.
Set
	\begin{equation}\label{Hn}
		\mathcal{H}_n \coloneqq 	\begin{pmatrix}
			t_p H^{(1)\prime}_n(t_p) & -in H^{(1)}_n(t_s)\\
			in H^{(1)}_n(t_p) & t_s H^{(1)\prime}_n(t_s)
		\end{pmatrix} \text{ and }
		Q \coloneqq \begin{pmatrix}
			\frac{1}{\sqrt{k_p}} & 0 \\
			0 & \frac{1}{\sqrt{k_s}}	
		\end{pmatrix}.
	\end{equation}
	From \eqref{uscR} and \eqref{Hn}, we have 
	\begin{equation}\label{scfield}
		u^{s}(r,\theta_x)|_{r=R} = \sum_{n \in \mathbb{Z}} 
		\begin{pmatrix}
			\nu & \tau
		\end{pmatrix}
		\mathcal{H}_n Q
		\begin{pmatrix}
			A_n \\ B_n
		\end{pmatrix}
		\frac{\text{e}^{in \theta_x}}{R}.
	\end{equation}

	\begin{remark}\label{hn-determin} By 
	\cite[Lemma 2.11]{Bao2018} and \cite[Remark 3.2]{Li2016}, 
	it follows that $|\mathcal{H}_n| = t_p t_s H_n^{(1)\prime}(t_p)H_n^{(1)\prime}(t_s) - n^2 H_n^{(1)}(t_p)H_n^{(1)}(t_s) \neq 0$ for any $n \in \mathbb{Z}$. Hence
the matrix $\mathcal{H}_n$ is always invertible.
\end{remark}

Having expanded the scattered field into the series (\ref{scfield}), now we need to represent an elastic plane wave in terms of the special functions on $|x|=R$.
Let $d =\left(\cos \theta_d, \sin \theta_d\right)^T\in \s$ be the incident direction. By the Jacobi-Anger expansion (see e.g., \cite[Formula (3.89)]{kress.1998}), we see
	\begin{equation}\label{JAe}
		\text{e}^{ikx \cdot d} = \sum_{n \in \mathbb{Z}} i^n J_n\left(k|x|\right) \text{e}^{in \theta}, \quad x \in \mathbb{R}^2,
	\end{equation}
	where $\theta = \theta_x - \theta_d$ denotes the angle between $\hat{x}$ and $d$.
	Let $d^\bot = (-\sin \theta_d, \cos \theta_d)^T$ be a vector perpendicular to $d$. Recalling the compressional part $u^{i}_p$ of the incident wave $u^{i}$, 
	\begin{equation}\label{uinp}
		u^{i}_p (x) = d \  \text{e}^{i k_p x \cdot d},
	\end{equation}
and inserting \eqref{JAe} to \eqref{uinp}, we get the form of $u^{i}_p$ over $\partial B_R$ as
	\begin{equation}\label{5.14}
		u^{i}_p|_{r=R} = \sum_{n \in \mathbb{Z}} d J_n(t_p) i^n \text{e}^{in \theta}.
	\end{equation}
Using the formulas 
	\begin{equation}
		\cos \theta = \frac{1}{2}\left( \text{e}^{i \theta} + \text{e}^{-i \theta}\right), \ \sin \theta = \frac{1}{2i} \left( \text{e}^{i \theta} - \text{e}^{-i \theta}\right)
	\end{equation}
	and the following properties of Bessel functions
	\begin{equation}
		2 J^{\prime}_n(x) = J_{n-1}(x) - J_{n+1}(x), \ J_{n-1}(x) + J_{n+1}(x) = \frac{2n}{x} J_n(x),
	\end{equation}
	we can get from (\ref{5.14}) that 
	\begin{equation}\label{5.16}
		\begin{split}
			u^{i}_p \cdot \nu |_{r=R} &= \frac{1}{2} \sum_{n \in \mathbb{Z}} \left(\text{e}^{i \theta} + \text{e}^{-i \theta}\right) J_n(t_p ) i^n \text{e}^{in \theta} \\
			&= \frac{1}{2} \sum_{n \in \mathbb{Z}} \left( J_n(t_p) i^n \text{e}^{i(n+1) \theta} + J_n(t_p) i^n \text{e}^{i(n-1) \theta} \right) \\
			&= \frac{1}{2} \sum_{n \in \mathbb{Z}} \left(J_n(t_p) i^{n-1} \text{e}^{in \theta} + J_n(t_p) i^{n+1} \text{e}^{in \theta} \right) \\
			&= \frac{1}{2} \sum_{n \in \mathbb{Z}} \left(J_{n-1}(t_p) - J_{n+1}(t_p) \right) i^{n-1} \text{e}^{in \theta} \\
			&= \sum_{n \in \mathbb{Z}} J^{\prime}_n(t_p) i^{n-1} \text{e}^{in \theta}
		\end{split}
	\end{equation}
	and 
	\begin{equation}\label{5.17}
		\begin{split}
			u^{i}_p \cdot \tau |_{r=R} &= \frac{1}{2i} \sum_{n \in \mathbb{Z}} \left(\text{e}^{-i \theta} - \text{e}^{i \theta}\right) J_n(t_p) i^n \text{e}^{in \theta} \\
			&= \frac{1}{2} \sum_{n \in \mathbb{Z}} \left( J_n(t_p) i^{n-1} \text{e}^{i(n-1) \theta} - J_n(t_p) i^{n-1} \text{e}^{i(n+1) \theta} \right) \\
			&= \frac{1}{2} \sum_{n \in \mathbb{Z}} \left(J_{n+1}(t_p) i^{n} \text{e}^{in \theta} + J_{n-1}(t_p) i^{n} \text{e}^{in \theta} \right) \\
			&= \sum_{n \in \mathbb{Z}} \frac{n}{t_p} J_n(t_p) i^{n} \text{e}^{in \theta}.
		\end{split}
	\end{equation}
	
We use the notation $v^{s}$ to denote the scattered field produced by the compressional part $u^{i}_p$ of the incident wave $u^{i}$. Note that $v^{s}$  is usually different from the compressional part of the scattered field for $u^{i}$.
By \eqref{scfield}, we can represent $v^{s}(r,\theta_x)$ on $|x|=R$	as the series
 \begin{equation}
		v^{s}(r,\theta_x)|_{r=R} =\sum_{n \in \mathbb{Z}} 
		\begin{pmatrix}
			\nu & \tau
		\end{pmatrix}
		\mathcal{H}_n Q
		\begin{pmatrix}
			A_{n,v} \\ B_{n,v}
		\end{pmatrix}
		\frac{\text{e}^{in \theta_x}}{R},
	\end{equation}
	where the coefficients $A_{n, v}, B_{n, v}\in\C$  are associated with $v^{s}$. Making use of the Dirichlet boundary condition $u^{i}_p = -v^{s}$ on $\partial B_R$, we have 
	\begin{equation*}
		\left.
		\begin{pmatrix}
			u^{i}_p \cdot \nu \\ u^{i}_p \cdot \tau
		\end{pmatrix} \right\rvert_{r=R} 
		= - \left.
		\begin{pmatrix}
			v^{s} \cdot \nu \\ v^{s} \cdot \tau
		\end{pmatrix} \right\rvert_{r=R},
	\end{equation*}
	which together with (\ref{5.16}) and (\ref{5.17}) leads to 
	\begin{equation}
		\sum_{n \in \mathbb{Z}}
		\begin{pmatrix}
			J^{\prime}_n(t_p) \\ \frac{in}{t_p} J_n(t_p)
		\end{pmatrix}
		i^{n-1} \text{e}^{in \theta_x} \text{e}^{-in \theta_d} = -\sum_{n \in \mathbb{Z}}
		\mathcal{H}_n Q
		\begin{pmatrix}
			A_{n,v} \\ B_{n,v}
		\end{pmatrix}
		\frac{\text{e}^{in \theta_x}}{R}. 
	\end{equation}
	By the arbitrariness of $\theta_x$, we can get
	\begin{equation*}
		\begin{pmatrix}
			i J^{\prime}_n(t_p) \\ -\frac{n}{t_p} J_n(t_p)
		\end{pmatrix}
		i^{n} \text{e}^{-in \theta_d} 
		= 
		\frac{\mathcal{H}_n Q}{R}
		\begin{pmatrix}
			A_{n,v} \\ B_{n,v}
		\end{pmatrix},
	\end{equation*}
	implying the relation 
	\begin{equation}
		\begin{pmatrix}
			A_{n,v} \\ B_{n,v}
		\end{pmatrix}
		= Q^{-1} \mathcal{H}_n^{-1}
		\begin{pmatrix}
			i J^{\prime}_n(t_p) \\ -\frac{n}{t_p} J_n(t_p)
		\end{pmatrix}
		R i^{n} \text{e}^{-in \theta_d}  \,.
	\end{equation}
	Thus, we obtain the P-part and S-part of the far-field patterns of $v^{s}$ as follows:
	\begin{equation}\label{ppsp-infty}
		\begin{pmatrix}
			u^{\infty}_{pp}(\hat{x}) \\ 
			u^{\infty}_{sp}(\hat{x})
		\end{pmatrix}
		= \sqrt{\frac{2}{\pi}} \text{e}^{i\frac{\pi}{4}}iR \sum_{n \in \mathbb{Z}} Q^{-1} \mathcal{H}_n^{-1} 
		\begin{pmatrix}
			 J^{\prime}_n(t_p) \\ \frac{in}{t_p} J_n(t_p)
		\end{pmatrix}
		\text{e}^{in \theta}. 
	\end{equation}

The far-field pattern excited by the S-part of an elastic plane wave can be treated similarly.
The shear part $u^{i}_s$ of the incident wave $u^{i}$ takes the form 
	\begin{equation}
		u^{i}_s (x) = d^{\bot } \text{e}^{i k_s x \cdot d},
	\end{equation}
which together with  \eqref{JAe} gives arise to
	\begin{equation}
		u^{i}_s|_{r=R} = \sum_{n \in \mathbb{Z}} d^{\bot } J_n(t_s) i^n \text{e}^{in \theta}.
	\end{equation}
	Correspondingly, we have
	\begin{equation*}
			u^{i}_s \cdot \nu |_{r=R} = -\sum_{n \in \mathbb{Z}} \frac{n}{t_s} J_n(t_s) i^{n} \text{e}^{in \theta}, \quad
			u^{i}_s \cdot \tau |_{r=R}= \sum_{n \in \mathbb{Z}} J^{\prime}_n(t_s) i^{n-1} \text{e}^{in \theta}.
	\end{equation*}
Again using \eqref{scfield}, we can represent by $w^{s}$ the scattered field produced by $u^{i}_s$ in the form of
	\begin{equation}
		w^{s}(r,\theta_x)|_{r=R} =\sum_{n \in \mathbb{Z}} 
		\begin{pmatrix}
			\nu & \tau
		\end{pmatrix}
		\mathcal{H}_n Q
		\begin{pmatrix}
			A_{n,w} \\ B_{n,w}
		\end{pmatrix}
		\frac{\text{e}^{in \theta_x}}{R}.
	\end{equation}
Combining the proves two identities together with the boundary condition $u^{i}_s = -w^{s}$ on $\partial B_R$, we arrive at  
	\begin{equation}
		\begin{pmatrix}
			A_{n,w} \\ B_{n,w}
		\end{pmatrix}
		= Q^{-1} \mathcal{H}_n^{-1}
		\begin{pmatrix}
			\frac{n}{t_s} J_n(t_s) \\ i J^{\prime}_n(t_s)
		\end{pmatrix}
		R i^{n} \text{e}^{-in \theta_d}.  
	\end{equation}
	Thus, we obtain $u^{\infty}_{ps}$ and $u^{\infty}_{ss}$ as follows:
	\begin{equation}\label{psss-infty}
		\begin{pmatrix}
			u^{\infty}_{ps}(\hat{x})\\ 
			u^{\infty}_{ss}(\hat{x})
		\end{pmatrix}
		= \sqrt{\frac{2}{\pi}} \text{e}^{i\frac{\pi}{4}}iR \sum_{n \in \mathbb{Z}} Q^{-1} \mathcal{H}_n^{-1} 
		\begin{pmatrix}
			-\frac{in}{t_s} J_n(t_s) \\  J^{\prime}_n(t_s)
		\end{pmatrix}
		\text{e}^{in \theta}. 
	\end{equation}
 This enables us to define the matrix
	\begin{equation}
		U_{B_R}^{\infty} \coloneqq
		\begin{pmatrix}
			u^{\infty}_{pp}(\hat{x}) & u^{\infty}_{ps}(\hat{x}) \\ 
			u^{\infty}_{sp}(\hat{x}) & u^{\infty}_{ss}(\hat{x})
		\end{pmatrix}
		= \sqrt{\frac{2}{\pi}} \text{e}^{i\frac{\pi}{4}} i \sum_{n \in \mathbb{Z}} Q^{-1} \mathcal{H}_n^{-1}
		\begin{pmatrix}
			R J^{\prime}_n(t_p) & -\frac{in}{k_s} J_n(t_s) \\
		    \frac{in}{k_p} J_n(t_p) & R J^{\prime}_n(t_s)
		\end{pmatrix}
		\text{e}^{in \theta}.
	\end{equation}
Setting
	\begin{equation}
		\mathcal{J}_n \coloneqq \begin{pmatrix}
			t_p J^{\prime}_n(t_p) & -in J_n(t_s)\\
			in J_n(t_p) & t_s J^{\prime}_n(t_s)
		\end{pmatrix},
	\end{equation}
we can rewrite $U_{B_R}^{\infty}$ as
	\begin{equation}\label{matrix-uinf}
		U_{B_R}^{\infty}(\hat{x}) = \sqrt{\frac{2}{\pi}} \text{e}^{i\frac{\pi}{4}} i \sum_{n \in \mathbb{Z}} Q^{-1} \mathcal{H}_n^{-1} \mathcal{J}_n Q^2
		\text{e}^{in \theta}.
	\end{equation}

To sum up,  for the elastic plane wave  $u^{i}$ of the general form (\ref{2.1}), by linear superposition its far-field pattern $u^{\infty}$ takes the form of 
	\eqref{u-infty}, where the P-part and S-part for $u^{i}_p$ and 
	$u^{i}_s$
	 are given in the matrix $U_{B_R}^\infty$ (see also \eqref{ppsp-infty} and \eqref{psss-infty}).

	\subsection{Spectral system of  the far-field operator $F_{B_R}$}\label{sec5.2}
	
	Now, we need to derive eigenvalues and the associated eigenfunctions  of the far-field operator $F_{B_R}$ defined by (e.g. (\ref{op-domain}))
	\begin{equation}\label{op-br}
			(F_{B_R}g)(\hat{x}) = \text{ e}^{-i\frac{\pi}{4}} \int_{\mathtt{S}} \left\{\sqrt{\frac{k_p}{\omega}}u^{\infty}_{B_R}(\hat{x}; d, 1, 0)g_p(d) + \sqrt{\frac{k_s}{\omega}}u^{\infty}_{B_R}(\hat{x}; d, 0, 1)g_s(d)\right\} ds(d).
		\end{equation}
Obviously, the spectral system of $F_{B_R}$ should be connected to the spectral system of the matrix $U_{B_R}$. To disclose this relation, we retain the notations from the previous subsection to define $\widetilde{\Sigma}_n  := Q^{-1} \mathcal{H}_n^{-1} \mathcal{J}_n Q$.

	\begin{lemma}\label{gn}
		If $(\lambda_n, \widetilde{X}_n)$ is the spectral pair of $\widetilde{\Sigma}_n$, that is, $\widetilde{\Sigma}_n\widetilde{X}_n = \lambda_n \widetilde{X}_n,\ \widetilde{X}_n = (\widetilde{X}_n^{(1)}, \widetilde{X}_n^{(2)})^T$. 
		Then 
		\begin{equation}\label{L5.2}
		(F_{B_R}g)(\hat{x}) = \sqrt{\frac{8\pi}{\omega}}i \lambda_n\; g(\hat{x}),\quad g(\hat{x}):=(\widetilde{X}_n^{(1)} \hat{x} + \widetilde{X}_n^{(2)} \hat{x}^{\bot })\text{e}^{in \theta_x}.
		\end{equation}
	\end{lemma}
	\begin{proof}  Let $g\in \left(L^2(\s)\right)^2$ be given as in \eqref{L5.2}.
It is easy to see the P- and S-component of $g$ as $g_p(d)=\widetilde{X}_n^{(1)} \text{e}^{in \theta_d}$, $g_s(d)=\widetilde{X}_n^{(2)} \text{e}^{in \theta_d}$. It then follows from the definition of $F_{B_R}$ in \eqref{op-br} that
	\begin{equation*}
		\begin{split}
			\left(F_{B_R}g\right)(\hat{x}) &= \frac{1}{\sqrt{\omega}} \text{ e}^{-i\frac{\pi}{4}}
			\int_{\mathbb{S}} g_p(d) \left[\sqrt{k_p}u^{\infty}_{pp}(\hat{x}; d) \hat{x} + \sqrt{k_p}u^{\infty}_{sp}(\hat{x}; d) \hat{x}^{\bot}\right] \\ 
			& \quad + g_s(d) \left[\sqrt{k_s}u^{\infty}_{ps}(\hat{x}; d) \hat{x} + \sqrt{k_s}u^{\infty}_{ss}(\hat{x}; d) \hat{x}^{\bot}\right] ds(d) \\
			&= \frac{1}{\sqrt{\omega}} \text{ e}^{-i\frac{\pi}{4}} \int_{\mathbb{S}} \left[\widetilde{X}_n^{(1)} \sqrt{k_p}u^{\infty}_{pp}(\hat{x}; d) \hat{x} + \widetilde{X}_n^{(1)} \sqrt{k_p}u^{\infty}_{sp}(\hat{x}; d) \hat{x}^{\bot}\right]\text{e}^{in \theta_d} \\ 
			& \quad + \left[\widetilde{X}_n^{(2)} \sqrt{k_s}u^{\infty}_{ps}(\hat{x}; d) \hat{x} + \widetilde{X}_n^{(2)} \sqrt{k_s}u^{\infty}_{ss}(\hat{x}; d) \hat{x}^{\bot}\right]\text{e}^{in \theta_d} ds(d) \\
			&= \frac{1}{\sqrt{\omega}} \text{ e}^{-i\frac{\pi}{4}} \int_{\mathbb{S}} \left[ \left(\widetilde{X}_n^{(1)} \sqrt{k_p}u^{\infty}_{pp}(\hat{x}; d) + \widetilde{X}_n^{(2)} \sqrt{k_s}u^{\infty}_{ps}(\hat{x}; d) \right)\hat{x}\right. \\
			& \quad + \left. \left(\widetilde{X}_n^{(1)} \sqrt{k_p}u^{\infty}_{sp}(\hat{x}; d) + \widetilde{X}_n^{(2)} \sqrt{k_s}u^{\infty}_{ss}(\hat{x}; d)\right)
			\hat{x}^{\bot}\right]\text{e}^{in \theta_d} ds(d) \\
			&= \frac{1}{\sqrt{\omega}} \text{ e}^{-i\frac{\pi}{4}} \int_{\mathbb{S}} \begin{pmatrix}
				\hat{x} & \hat{x}^{\bot}
			\end{pmatrix}
			U_{B_R}^{\infty} Q^{-1} \widetilde{X}_n \text{e}^{in \theta_d} ds(d)\\
			&= \sqrt{\frac{2}{\omega \pi}}i \int_{\mathbb{S}} \sum_{m \in \mathbb{Z}} \begin{pmatrix}
				\hat{x} & \hat{x}^{\bot}
			\end{pmatrix} \widetilde{\Sigma}_m \widetilde{X}_n \text{e}^{im \theta_x} \text{e}^{in \theta_d} ds(d).
		\end{split}
	\end{equation*}
Using the orthogonality of $\text{e}^{in \theta_d}$ for $n\in \Z$ and the fact that  $\widetilde{\Sigma}_n$ and $\widetilde{X}_n$ are independent of $d$, we arrive at
	\begin{equation*}
		\begin{split}
			\left(F_{B_R}g\right)(\hat{x}) &= \sqrt{\frac{2}{\omega \pi}}i \int_{\mathbb{S}} \begin{pmatrix}
				\hat{x} & \hat{x}^{\bot}
			\end{pmatrix} \widetilde{\Sigma}_n \widetilde{X}_n \text{e}^{in \theta_x} ds(d)\\
			&= \sqrt{\frac{8\pi}{\omega}}i  \begin{pmatrix}
				\hat{x} & \hat{x}^{\bot}
			\end{pmatrix} \widetilde{\Sigma}_n \widetilde{X}_n \text{e}^{in \theta_x} \\
			&= \sqrt{\frac{8\pi}{\omega}}i  \begin{pmatrix}
				\hat{x} & \hat{x}^{\bot}
			\end{pmatrix} \lambda_n \widetilde{X}_n \text{e}^{in \theta_x}\\
			&= \sqrt{\frac{8\pi}{\omega}}i \lambda_n g(\hat{x}).
		\end{split}
	\end{equation*}
    \end{proof}
As a consequence of Lemma \ref{gn}, we obtain the spectral pair of $F_{B_R}$ as follows. 
	\begin{lemma}\label{eig-FBR}
		The spectral pair of $F_{B_R}$ is given by
	\[
	(\lambda^{(n)}_{B_R}, X^{(n)}_{B_R}) = \left(\sqrt{\frac{8\pi}{\omega}}i \lambda_n,\; \begin{pmatrix}
			\hat{x} & \hat{x}^{\bot }
		\end{pmatrix}
		 Q^{-1}X_n \text{e}^{in \theta_x} \right),
		 \]
		  where $(\lambda_n, X_n)$ is the spectral pair of $\Sigma_n \coloneqq \mathcal{H}_n^{-1} \mathcal{J}_n$.
	\end{lemma}
	\begin{proof} Suppose that $(\lambda_n, \widetilde{X}_n)$ is the spectral pair of $\widetilde{\Sigma}_n$.	Writing $X_n := Q \widetilde{X}_n$, we have
		\begin{equation}
			\lambda_n \widetilde{X}_n =\widetilde{\Sigma}_n \widetilde{X}_n = Q^{-1}\Sigma_n Q \widetilde{X}_n = Q^{-1}\Sigma_n X_n.
		\end{equation}
This implies that $\lambda_n  X_n = \Sigma_n  X_n$.
Using Lemma \ref{gn}, we get
\[
\lambda^{(n)}_{B_R}= \sqrt{\frac{8\pi}{\omega}}i \lambda_n,
\quad
X^{(n)}_{B_R}(\hat{x})=\left(\hat{x} \;\;  \hat{x}^{\bot }\right) Q^{-1}X_n \text{e}^{in \theta_x}.
\]
	\end{proof}
Since the eigenvalues of $F_{B_R}$ have appeared in the denominator of the indicator \eqref{svd-omega} with $\Omega=B_R$, 
it is necessary to show  $\lambda^{(n)}_{B_R}\neq 0$ for all $n\in \N$ under an additional assumption of the frequency. 
	\begin{lemma}
		If $\omega^2$ is not a Dirichlet eigenvalue of $-\Delta^{\ast}$ in $B_R$, then the eigenvalue $\lambda_n$ of the matrix $\Sigma_n$ cannot vanish for any $n \in \mathbb{Z}$.
	\end{lemma}
	\begin{proof}
		Suppose that there exists  $n\in \mathbb{Z}$ such that $\lambda_n = 0$ is the eigenvalue of the matrix $\Sigma_n$ and that $X_n \neq 0$ is the corresponding eigenvector. Then, 
		\begin{equation*}
			\mathcal{H}_n^{-1} \mathcal{J}_n X_n=\Sigma_n X_n = \lambda_n X_n = 0,
		\end{equation*}
		Since $\mathcal{H}_n^{-1}$ is invertible
		(see Remark \ref{hn-determin}), 
		we have 
			$\mathcal{J}_n X_n = 0$, implying that
$|\mathcal{J}_n|=0$.  Since
$X_n = (X_n^{(1)},X_n^{(2)})^T \neq 0$, we may define the non-trivial  function $u = \grad \varphi + \curl \psi$ where
\[
\varphi = \sum_{n \in \mathbb{Z}}  X_n^{(1)} J_n(k_p r)\text{e}^{in\theta_x}, \quad \psi = \sum_{n \in \mathbb{Z}} X_n^{(2)} J_n(k_s r)\text{e}^{in\theta_x}.
\]  Then it is easy to check that
		\begin{equation*}
			\begin{split}
				& \quad u|_{\partial B_R} \\
				&= \sum_{n \in \mathbb{Z}} \frac{\text{e}^{in\theta_x}}{R} \left[
					\left(
						k_p R J^{\prime}_n(k_p R) X_n^{(1)} - in J_n(k_s R)X_n^{(2)}
					\right) \hat{x} + 
					\left(
						in J_n(k_p R)X_n^{(1)} + k_s R J^{\prime}_n(k_s R)X_n^{(2)}
					\right) \hat{x}^{\bot}\right]\\
				&= \sum_{n \in \mathbb{Z}} \frac{\text{e}^{in\theta_x}}{R} 
				\begin{pmatrix}
					\hat{x} & \hat{x}^{\bot}
				\end{pmatrix}
				\begin{pmatrix}
					t_p J^{\prime}_n(t_p) & - in J_n(t_s ) \\
					in J_n(t_p) & t_s J^{\prime}_n(t_s)
				\end{pmatrix}
				\begin{pmatrix}
					X_n^{(1)} \\ X_n^{(2)}
				\end{pmatrix}\\
				&= \sum_{n \in \mathbb{Z}} \frac{\text{e}^{in\theta_x}}{R} 
				\begin{pmatrix}
					\hat{x} & \hat{x}^{\bot}
				\end{pmatrix}
				\mathcal{J}_n X_n\\
				&= 0.
			\end{split}
		\end{equation*}
	On the other hand, it is obvious that  $u$ satisfies the Navier equation $-\Delta^{\ast} u = \omega^2 u $ in  $B_R$. Hence, it is a Dirichlet eigenfunction of $-\Delta^*$ over $B_R$, which is impossible.
	\end{proof}

To calculate the spectra of $F_{B_R}$, by Lemma \ref{eig-FBR} we need to consider the generalized eigenvalue problem 
	\begin{equation}\label{gev}
		\mathcal{J}_n X_n = \lambda_n \mathcal{H}_n X_n,
	\end{equation}
	where $\lambda_n$ and $X_n$ represent eigenvalues and eigenvectors of $\Sigma_n$.
Recalling the Hankel functions and its derivatives,
	\begin{equation*}
		H^{(1)}_n(z) = J_n(z) + i Y_n(z), \  H^{(1)\prime}_n(z) = J^{\prime}_n(z) + i Y^{\prime}_n(z),
	\end{equation*}
and  setting
	\begin{equation*}
		\mathcal{Y}_n \coloneqq \begin{pmatrix}
			t_p Y^{\prime}_n(t_p) & -in Y_n(t_s)\\
			in Y_n(t_p) & t_s Y^{\prime}_n(t_s)
		\end{pmatrix},
	\end{equation*}
we can rephrase the matrix $\mathcal{H}_n$ as
	\begin{equation}\label{jyh}
		\mathcal{H}_n = \mathcal{J}_n + i \mathcal{Y}_n.
	\end{equation}
Below we describe an eigensystem of the generalized eigenvalue problem (\ref{gev}) with the help of the decomposition \eqref{jyh}.

	\begin{lemma}\label{eig-hn}
A normalized eigensystem $(\lambda_{n,j}, X_{n,j})$ with $n\in \N, j=1,2$ to the generalized eigenvalue problem (\ref{gev}) is given by 
\begin{equation}\label{5.34}
\lambda_{n,j} = \frac{t_p J_n^{\prime}(t_p) + in J_n(t_s)\sigma_j^{(n)}}{t_p H_n^{(1)\prime}(t_p) + in H_n^{(1)}(t_s)\sigma_j^{(n)}}, \quad
X_{n,j} = \frac{(1,\sigma_j^{(n)})^T}{\sqrt{1+|\sigma_j^{(n)}|^2}},
\end{equation}	
with
\begin{equation*}
\begin{split}
&\sigma_1^{(n)} = \frac{-\beta_n + \sqrt{\beta_n^2 -4}}{2}, \quad
		\sigma_2^{(n)} = \frac{-\beta_n - \sqrt{\beta_n^2 -4}}{2},\\
&\beta_n = \frac{\pi}{2in}\left[n^2(J_n(t_s)Y_n(t_p)-J_n(t_p)Y_n(t_s))+t_pt_s(J_n^{\prime}(t_p)Y_n^{\prime}(t_s)-J_n^{\prime}(t_s)Y_n^{\prime}(t_p))\right].
\end{split}
\end{equation*}
		
	\end{lemma}

	\begin{proof}
		Let $X_n = (1,\sigma^{(n)})^T$ be an eigenvector of the generalized  eigenvalue problem $\mathcal{J}_n X_n = \eta_n \mathcal{Y}_n X_n$, where $\eta_n$ is the eigenvalue. 
Using the Wronskian 
		\begin{equation*}
			J_n(t)Y_n^{\prime}(t)-J_n^{\prime}(t)Y_n(t) = \frac{2}{\pi t},
		\end{equation*}
simple calculations show that $\sigma^{(n)}$ should satisfy the algebraic equation
		\begin{equation}\label{x2}
			\sigma^{(n)2}+ \beta_n \sigma^{(n)} +1 =0,
		\end{equation}
		where $\beta_n$ is defined as in the lemma.
The two roots of \eqref{x2} are given by
		\begin{equation*}
			\sigma_1^{(n)} = \frac{-\beta_n + \sqrt{\beta_n^2 -4}}{2}, \quad
			\sigma_2^{(n)} = \frac{-\beta_n - \sqrt{\beta_n^2 -4}}{2}.
		\end{equation*}
On the other hand, one can also calculate the corresponding eigenvalues 
		\begin{equation*}
			\eta_{n,j} = \frac{t_p J_n^{\prime}(t_p) + in J_n(t_s)\sigma_j^{(n)}}{t_p Y_n^{\prime}(t_p) + in Y_n(t_s)\sigma_j^{(n)}},\quad  j=1,2.
		\end{equation*}
Using the decomposition (\ref{jyh}), we get
		\begin{equation*}
			\mathcal{H}_n X_n = (\mathcal{J}_n + i \mathcal{Y}_n) X_n = \frac{\eta_n+i}{\eta_n} \mathcal{J}_n X_n.
		\end{equation*}
Therefore,  $(\frac{\eta_n}{\eta_n+i},X_n)$ is the eigensystem of the generalized eigenvalue problem $\mathcal{J}_n X_n = \lambda_n \mathcal{H}_n X_n$.
		Further, we get the eigenvalues 
		\begin{equation}
			\lambda_{n,j} =\frac{\eta_{n,j}}{\eta_{n,j}+i}= \frac{t_p J_n^{\prime}(t_p) + in J_n(t_s)\sigma_j^{(n)}}{t_p H_n^{(1)\prime}(t_p) + in H_n^{(1)}(t_s)\sigma_j^{(n)}}, \  j=1,2.
		\end{equation}
	\end{proof}
Combining Lemma \ref{eig-FBR} and \ref{eig-hn}, we obtain an eigensystem of the far-field operator $F_{B_R}$ by
	\begin{equation}\label{eigsys-fbr}
		\begin{split}
			\lambda^{(n)}_{B_R,j} &= \sqrt{\frac{8\pi}{\omega}}\frac{it_p J_n^{\prime}(t_p) - n J_n(t_s)\sigma_j^{(n)}}{t_p H_n^{(1)\prime}(t_p) + in H_n^{(1)}(t_s)\sigma_j^{(n)}}, \\
			X^{(n)}_{B_R,j} &= \left(\sqrt{k_p}X_{n,j}^{(1)} \hat{x} + \sqrt{k_s}X_{n,j}^{(2)} \hat{x}^{\bot} \right) \text{e}^{in\theta_x}\\
			&= \left(\sqrt{k_p}\hat{x}+\sqrt{k_s}\sigma_j^{(n)}\hat{x}^{\bot}\right)\frac{\text{e}^{in \theta_x}}{\sqrt{1+|\sigma_j^{(n)}|^2}}
		\end{split}
	\end{equation}
for $n\in \N, j=1,2$. 

\begin{remark}\label{asy-beta}
The asymptotics of $\sigma_j^{(n)}$ can be derived as follows.
Recall the asymptotic behavior of 	
Bessel functions (see \cite{kress.1998}) 
	\begin{equation}
		\begin{split}
			J_{n}(z) &= \frac{z^{n}}{2^{n} n!} \left(1+\mathcal{O}(\frac{1}{n})\right),\quad n \rightarrow +\infty, \\
			J_{n}^{\prime}(z) &= \frac{z^{n-1}}{2^{n} (n-1)!} \left(1+\mathcal{O}(\frac{1}{n})\right),\quad n \rightarrow +\infty,
		\end{split}
	\end{equation}
	and those of Neumann functions:
	\begin{equation}
		\begin{split}
			Y_{n}(z) &= -\frac{2^{n}(n-1)!}{\pi z^{n}}
			\left(1+\mathcal{O}(\frac{1}{n})\right),\quad n \rightarrow +\infty,\\
			Y_{n}^{\prime}(z) &= \frac{2^{n} n!}{\pi z^{n+1}} 
			\left(1+\mathcal{O}(\frac{1}{n})\right),\quad n \rightarrow +\infty.
		\end{split}
	\end{equation}
Then we get from the definition of $\beta_n$ stated in Lemma \ref{eig-hn} that
\begin{equation*}\begin{split}
\beta_n = i\left(\frac{t_s^n}{t_p^n}-\frac{t_p^n}{t_s^n}\right) \left(1+\mathcal{O}\left(\frac{1}{n}\right)\right),\quad
\sqrt{\beta_n^2-4} = i\left(\frac{t_s^n}{t_p^n}+\frac{t_p^n}{t_s^n}\right) \left(1+\mathcal{O}\left(\frac{1}{n}\right)\right),
\end{split}
\end{equation*}
whence it follows as $n\to\infty$ that , 
\[ 
\sigma_1^{(n)}=i\frac{t_p^n}{t_s^n}\left(1+\mathcal{O}\left(\frac{1}{n}\right)\right),\quad \sigma_2^{(n)}=-i\frac{t_s^n}{t_p^n}\left(1+\mathcal{O}\left(\frac{1}{n}\right)\right).
\]
\end{remark}

Now we can get the asymptotic behavior of the eigenvalues of $F_{B_R}$ as $n\rightarrow\infty$.
Using \eqref{eigsys-fbr}, Remark \ref{asy-beta} and the following recurrence relations
	\begin{equation*}
		\begin{split}
			tJ_n^{\prime}(t) = nJ_n(t) - tJ_{n+1}(t), \quad
			tH_n^{(1)\prime}(t) = tH_{n-1}^{(1)}(t) - nH_n^{(1)}(t),
		\end{split}
	\end{equation*} 
we find
	\begin{equation}\label{FBR-VAL}
		\begin{split}
		\lambda^{(n)}_{B_R,1} &= \sqrt{\frac{8\pi}{\omega}}\frac{inJ_n(t_p) - it_pJ_{n+1}(t_p) - n J_n(t_s)\sigma_1^{(n)}}{t_p H_n^{(1)\prime}(t_p) + in H_n^{(1)}(t_s)\sigma_1^{(n)}} \\
		&= -\sqrt{\frac{2\pi}{\omega}} \frac{\pi  t_p^{2n+2} t_s^{2n}}{2^{2n} (n+1)! n! (t_p^{2n} + t_s^{2n})} \left(1+\mathcal{O}(\frac{1}{n})\right),\\ 
		\lambda^{(n)}_{B_R,2} &= \sqrt{\frac{8\pi}{\omega}}\frac{it_p J_n^{\prime}(t_p) - n J_n(t_s)\sigma_2^{(n)}}{t_p H_{n-1}^{(1)}(t_p) - nH_n^{(1)}(t_p) + in H_n^{(1)}(t_s)\sigma_2^{(n)}} \\
		&= -\sqrt{\frac{2\pi}{\omega}} \frac{\pi   (t_p^{2n} + t_s^{2n})}{2^{2n-2} (n-1)! (n-2)! t_p^2} \left(1+\mathcal{O}\left(\frac{1}{n}\right)\right).
		\end{split}
	\end{equation}
Further, from (\ref{5.34}) we get the asymptotics of the eigenvectors of $\Sigma_n$ as follows	
\begin{equation}\label{FBR-VEC}
		\begin{split}
			X_{n,1} &= (X_{n,1}^{(1)},X_{n,1}^{(2)})^T = \left(\frac{t_s^{n}}{\sqrt{t_p^{2n} + t_s^{2n}}}, \frac{i t_p^{n}}{\sqrt{t_p^{2n} + t_s^{2n}}}\right)^T  \; \left(1+\mathcal{O}\left(\frac{1}{n}\right)\right) ,\\
			X_{n,2} &= (X_{n,2}^{(1)},X_{n,2}^{(2)})^T = \left(\frac{ t_p^{n}}{\sqrt{t_p^{2n} + t_s^{2n}}}, \frac{-i t_s^{n}}{\sqrt{t_p^{2n} + t_s^{2n}}}\right)^T\;  \left(1+\mathcal{O}\left(\frac{1}{n}\right)\right)  .
		\end{split}
	\end{equation}

	\subsection{Proof of Corollaries \ref{fm_convex_domain} and  \ref{fm_convex_domainpp}  for testing disks}\label{sec5.3}
In this subsection, we will use the eigensystem $(\lambda_{B_R, j}^{(n)}, X_{B_R, j}^{(n)})$  for $n\in \N, j=1,2$ (see \eqref{eigsys-fbr}) of the far-field operator $F_{B_R}$ to verify Corollaries 
\ref{fm_convex_domain} and  \ref{fm_convex_domainpp} with $\Omega = B_R$.   Corollary \ref{fm_convex_domain} can be rephrased as
	\begin{corollary}\label{fmdiskff}
		Let $v^\infty\in \left(L^2(\s)\right)^2$ and assume that
		$\omega^2$ is not a Dirichlet eigenvalue of $-\Delta^{\ast}$ over $B_R$. Then 
	   \begin{equation}
		   I(B_R) = \sum_{n \in \mathbb{Z}}\sum_{j=1}^2 \frac{\left|\left\langle v^{\infty}, X^{(n)}_{B_R, j}\right\rangle_{\mathbb{S}}\right|^{2}}{\left|\lambda^{(n)}_{B_R, j}\right|} < + \infty
	   \end{equation}
		   if and only if $v^{\infty}$ is the far-field pattern of some Kupradze radiating solution $v^{s}$, where $v^{s}$ satisfies the Navier equation
	   \begin{equation}\label{Navier-disk}
		   \Delta^{\ast} v^{s} + \omega^{2} v^{s} = 0 \qquad \text{in} \quad  \mathbb{R}^2 \backslash \overline{B_R},
	   \end{equation}
		with the boundary data $v^{s} |_{\partial B_R}\in \left(H^{1/2}(\partial B_R)\right)^2.$
	\end{corollary}
	
	\begin{proof} Let $v^s$ be a Kupradze radiating solution to (\ref{Navier-disk}). 
By the Hodge decomposition \eqref{Hodgedecom}, we may decompose $v^s$ into its compressional and shear parts by $v^s = \grad \varphi + \curl \psi$, where $\varphi = \sum_{n \in \mathbb{Z}}  a_n H^{(1)}_n(k_p r)\text{e}^{in\theta_x}$ and $\psi = \sum_{n \in \mathbb{Z}} b_n H^{(1)}_n(k_s r)\text{e}^{in\theta_x}$. Straightforward calculations lead to 
	\begin{equation}
		\begin{split}\label{vs}
			v^s &= \sum_{n \in \mathbb{Z}} \frac{\text{e}^{in\theta_x}}{r} \left[
				\left(
					k_p r H^{(1)\prime}_n(k_p r) a_n - in H^{(1)}_n(k_s r)b_n
				\right) \hat{x} + 
				\left(
					in H^{(1)}_n(k_p r)a_n + k_s r H^{(1)\prime}_n(k_s r)b_n
				\right) \hat{x}^{\bot}\right]\\
			&= \sum_{n \in \mathbb{Z}} \frac{\text{e}^{in\theta_x}}{r} 
			\begin{pmatrix}
				\hat{x} & \hat{x}^{\bot}
			\end{pmatrix}
			\begin{pmatrix}
				k_p r H^{(1)\prime}_n(k_p r) & - in H^{(1)}_n(k_s r) \\
				in H^{(1)}_n(k_p r) & k_s r H^{(1)\prime}_n(k_s r)
			\end{pmatrix}
			\begin{pmatrix}
				a_n \\ b_n
			\end{pmatrix}.
		\end{split}
	\end{equation}
	By the asymptotic behaviour of Hankel functions (see \cite{kress.1998})
	\begin{equation}
		\begin{split}
			H_{n}^{(1)}(z) &= \frac{2^{n}(n-1)!}{\pi iz^{n}}
			\left(1+\mathcal{O}(\frac{1}{n})\right),\quad n \rightarrow +\infty,\\
			H_{n}^{(1)\prime}(z) &= -\frac{2^{n} n!}{\pi iz^{n+1}} 
			\left(1+\mathcal{O}(\frac{1}{n})\right),\quad n \rightarrow +\infty,
		\end{split}
	\end{equation}
	we have 
	\begin{equation}
		\begin{split}
			v^s(r,\theta_x) &= \sum_{n \in \mathbb{Z}} \frac{\text{e}^{in\theta_x}}{r} 
			\begin{pmatrix}
				\hat{x} & \hat{x}^{\bot}
			\end{pmatrix}
			\begin{pmatrix}
				k_p r \frac{-2^n n!}{\pi i k_p^{n+1} r^{n+1} } & - in \frac{2^n (n-1)!}{\pi i k_s^nr^n} \\
				in \frac{2^n (n-1)!}{\pi i k_p^nr^n} & k_s r \frac{-2^n n!}{\pi i k_s^{n+1} r^{n+1} }
			\end{pmatrix}
			\begin{pmatrix}
				a_n \\ b_n
			\end{pmatrix} \left(1+\mathcal{O}(\frac{1}{n})\right)\\
			&= \sum_{n \in \mathbb{Z}} \frac{\text{e}^{in\theta_x}}{r} \frac{2^n n!}{\pi r^{n} }
			\begin{pmatrix}
				\hat{x} & \hat{x}^{\bot}
			\end{pmatrix}
			\begin{pmatrix}
				i k_p^{-n}  & -  k_s^{-n} \\
				k_p^{-n} & i k_s^{-n} 
			\end{pmatrix}
			\begin{pmatrix}
				a_n \\ b_n
			\end{pmatrix} \left(1+\mathcal{O}(\frac{1}{n})\right).
		\end{split}
	\end{equation}
This gives the leading term of the $H^{1/2}$-norm on $\partial B_R$ as
\begin{equation}\label{vs_Hhalfnorm}
		\begin{split}
			||v^s||^2_{\left(H^{1/2}(\partial B_R)\right)^2} &= \sum_{n \in \mathbb{Z}} (1+n^2)^{1/2} \frac{1}{R^2} \frac{2^{2n} n! n! }{\pi^2  R^{2n}} 2 \left| i k_p^{-n} a_n - k_s^{-n} b_n\right|^2  \left(1+\mathcal{O}(\frac{1}{n})\right)\\
			&\sim  \sum_{n \in \mathbb{Z}}  \frac{1}{R^2} \frac{2^{2n+1} (n+1)! n! }{\pi^2 k_p^n k_s^n R^{2n}}  \left| i \left(\frac{k_s}{k_p}\right)^{n/2} a_n - \left(\frac{k_p}{k_s}\right)^{n/2} b_n\right|^2  \left(1+\mathcal{O}(\frac{1}{n})\right)\\
			&=  \sum_{n \in \mathbb{Z}} \frac{2^{2n+1} (n+1)! n! }{\pi^2 R^2 t_p^n t_s^n }  \left| i \left(\frac{t_s}{t_p}\right)^{n/2} a_n - \left(\frac{t_p}{t_s}\right)^{n/2} b_n\right|^2  \left(1+\mathcal{O}(\frac{1}{n})\right)\\
			&=  \sum_{n \in \mathbb{Z}} C_n^{(1)} \frac{2^{2n+1} (n+1)! n! }{\pi^2 R^2 t_p^n t_s^n } \left(1+\mathcal{O}(\frac{1}{n})\right),
		\end{split}
	\end{equation}
	where $$C_n^{(1)} :=  \left| \left(\frac{t_s}{t_p}\right)^{n/2} a_n + \left(\frac{t_p}{t_s}\right)^{n/2}i b_n\right|^2.$$
On the other hand, the far-field pattern of $v^s$ can be calculated as 
	\begin{equation}
		v^{\infty}(\hat{x}) = \sqrt{\frac{2k_p}{\pi}}\text{e}^{i\frac{\pi}{4}} \sum_{n \in \mathbb{Z}} i^{-n} a_n \text{e}^{in\theta_x} \hat{x} + \sqrt{\frac{2k_s}{\pi}}\text{e}^{i\frac{\pi}{4}} \sum_{n \in \mathbb{Z}} i^{-n} b_n \text{e}^{in\theta_x} \hat{x}^{\bot}.
	\end{equation}
	We proceed with the proof by computing the leading term of the  the indicator $I(B_R)$. 
Using \eqref{eigsys-fbr} and \eqref{FBR-VEC},  the inner product over $L^2(\s)^2$ can be calculated as
	\begin{equation}
		\begin{split}
			&\quad \left|\left\langle v^{\infty},X_{B_R,j}^{(n)}\right\rangle_{\SS} \right|^2 \\
			&= \left| \int_{\SS} \sqrt{\frac{2}{\pi}} k_p \text{e}^{i\frac{\pi}{4}} \sum_{m \in \mathbb{Z}} i^{-m} a_m \text{e}^{im\theta_x} X_{n,j}^{(1)}\text{e}^{-in\theta_x} + \sqrt{\frac{2}{\pi}}k_s \text{e}^{i\frac{\pi}{4}} \sum_{m \in \mathbb{Z}} i^{-m} b_m \text{e}^{im\theta_x} X_{n,j}^{(2)}\text{e}^{-in\theta_x} d\theta_x \right|^2 \\
			&= \left|  \sqrt{\frac{2}{\pi}}  \text{e}^{i\frac{\pi}{4}} i^{-n} \int_{\SS} k_p a_n X_{n,j}^{(1)} + k_s b_n X_{n,j}^{(2)} d\theta_x \right|^2\\
			&= 8\pi \left| k_p a_n X_{n,j}^{(1)} + k_s b_n X_{n,j}^{(2)} \right|^2,
		\end{split}
	\end{equation}
	for $j=1,2$.
Using asymptotic behavior shown in (\ref{FBR-VAL}) and (\ref{FBR-VEC}), it is easy to check that, as $n \to \infty$,
\[
\left|\lambda^{(n)}_{B_R,1}\right| \sim \left|\lambda^{(n)}_{B_R,2} \right|\;n^{-4}, \quad 
\left|\left\langle v^{\infty},X_{B_R,1}^{(n)}\right\rangle_{\SS} \right|^2 \sim \left|\left\langle v^{\infty},X_{B_R,2}^{(n)}\right\rangle_{\SS} \right|^2.
\] 
	Thus, 
	\begin{equation}
		\begin{split}\label{IBR}
		I(B_R) &= \sum_{n \in \mathbb{Z}} \left(\frac{\left|\left\langle v^{\infty}, X^{(n)}_{B_R,1}\right\rangle_{\mathbb{S}}\right|^{2}}{\left|\lambda^{(n)}_{B_R,1}\right|} 
		+ \frac{\left|\left\langle v^{\infty}, X^{(n)}_{B_R,2}\right\rangle_{\mathbb{S}}\right|^{2}}{\left|\lambda^{(n)}_{B_R,2}\right|}\right)
		\\ 
		&= \sum_{n \in \mathbb{Z}} \frac{8\pi \left| k_p a_n X_{n,1}^{(1)} + k_s b_n X_{n,1}^{(2)} \right|^2}{\left|\lambda^{(n)}_{B_R,1}\right|}\left(1+\mathcal{O}(\frac{1}{n})\right) 	\\
		&= \sum_{n \in \mathbb{Z}}  C_{n}^{(2)} \sqrt{\frac{\omega}{2\pi}} \frac{2^{2n+3} (n+1)! n!}{t_p^{n}  t_s^{n} R^2 k_p^2 }  \left(1+\mathcal{O}(\frac{1}{n})\right),	\\
		\end{split}
	\end{equation}
	where $$C_{n}^{(2)} :=  \left| k_p \left(\frac{t_s}{t_p}\right)^{n/2} a_n + k_s \left(\frac{t_p}{t_s}\right)^{n/2}i b_n\right|^2.$$
Noting that $\min \{ k_p^2,k_s^2 \}C_{n}^{(1)}\leqslant C_{n}^{(2)} \leqslant \max \{k_p^2,k_s^2 \}C_{n}^{(1)}$, 
we conclude that the series \eqref{vs_Hhalfnorm} and \eqref{IBR} have the same convergence.

Since the boundedness of $||v^s||_{\left(H^{1/2}(\partial B_R)\right)^2}$ implies that $v^s$ of the form \eqref{vs} is indeed a radiating solution in $\mathbb{R}^2 \backslash \overline{B_R}$ with the far-field pattern $v^\infty$. This proves that $I(B_R)<\infty$ if and only if $v^s$ is a Kupradze radiating solution in $\mathbb{R}^2 \backslash \overline{B_R}$ with the far-field pattern $v^\infty$ and with the $H^{1/2}$-boundary data on $\partial B_R$.  The proof of Corollaries \ref{fmdiskff} is thus complete.
	\end{proof}

To prove Corollary \ref{fm_convex_domainpp} for testing disks, we need to consider spectral systems of the far-field operators
 $F_{B_R}^{(p)}$ and $F_{B_R}^{(s)}$ .  By Definition \ref{def-ps}, it follows that
 \begin{equation*}
 \begin{split}
( F_{B_R}^{(p)} g_p) (\hat{x}):=\text{e}^{-\frac{i\pi}{4}}\sqrt{\frac{k_p}{\omega}} \int_{\s} 
 u^{\infty}_{B_R,pp}(\hat{x};d) \, g_p(d)\, ds(d),\quad g_p\in L_p^2(\s),\\
 ( F_{B_R}^{(s)} g_s) (\hat{x}):=\text{e}^{-\frac{i\pi}{4}}\sqrt{\frac{k_s}{\omega}} \int_{\s} 
 u^{\infty}_{B_R,ss}(\hat{x}; d) \, g_s(d)\, ds(d),\quad g_s\in L_s^2(\s). \end{split}
  \end{equation*}
Using (\ref{u-inftypp}) and \eqref{matrix-uinf}, we see 
	\begin{equation*}
		u^{\infty}_{B_R,pp}(\hat{x}) = \sqrt{\frac{2}{\pi k_p}} \text{e}^{i\frac{\pi}{4}} i \sum_{n \in \mathbb{Z}} \Sigma_n(1,1) \text{e}^{in\theta},\quad
		u^{\infty}_{B_R,ss}(\hat{x}) = \sqrt{\frac{2}{\pi k_s}} \text{e}^{i\frac{\pi}{4}} i \sum_{n \in \mathbb{Z}} \Sigma_n(2,2) \text{e}^{in\theta},
	\end{equation*}
	where $\Sigma_n(i,j)$ dentoes the $(i,j)$-th entry of the matrix $\Sigma_n$.
Now we can get the spectral systems of the operators $F_{B_R}^{(p)}$ and $F_{B_R}^{(s)}$:
	\begin{equation}\label{fbr-eigpp}
		\begin{split}
			\eta_{B_R,p}^{(n)} &= \sqrt{\frac{2}{\pi k_p}} \text{e}^{i\frac{\pi}{4}} i \Sigma_n(1,1),\quad
			\varphi_{B_R,p}^{(n)}(\hat{x})= \text{e}^{in\theta_x},\\
			\eta_{B_R,s}^{(n)} &= \sqrt{\frac{2}{\pi k_s}} \text{e}^{i\frac{\pi}{4}} i \Sigma_n(2,2),\quad
			\varphi_{B_R,s}^{(n)}(\hat{x}) = \text{e}^{in\theta_x}.
		\end{split}
	\end{equation}
	Taking $\Omega = B_R$, we can rewrite Corollary \ref{fm_convex_domainpp} as 
	\begin{corollary}
		Let $w^\infty_{\alpha\alpha} \in L^2_{\alpha}(\s)$ ($\alpha = p, s$) and assume that
		$\omega^2$ is not a Dirichlet eigenvalue of $-\Delta^{\ast}$ over $B_R$. Denote by $(\lambda^{(n)}_{B_R,\alpha}, \varphi^{(n)}_{B_R,\alpha})$ a spectral system of the positive operator $F^{(\alpha)}_{B_R,\#} $. Then 
	   \begin{equation}
		   I^{(\alpha)}(B_R) = \sum_{n \in \mathbb{Z}} \frac{\left|\left\langle w^\infty_{\alpha\alpha}, \varphi^{(n)}_{B_R, \alpha}\right\rangle_{\mathbb{S}}\right|^{2}}{\left|\lambda^{(n)}_{B_R, \alpha}\right|} < + \infty
	   \end{equation}
		   if and only if $w^\infty_{\alpha\alpha}$ is the far-field pattern of some Sommerfeld radiating solution $w^s_{\alpha\alpha}$ which is defined in $\mathbb{R}^2 \backslash \overline{B_R}$ and $w^s_{\alpha\alpha} |_{\partial B_R(z)}\in H^{-1/2}(\partial B_R).$ That is,  $w^s_{\alpha\alpha}$ satisfies the following boundary value problem of the Helmholtz equation
	   \begin{equation}\label{Helmholtz-disk}
		   \Delta w^s_{\alpha\alpha} + k_{\alpha}^{2} w^s_{\alpha\alpha} = 0 \  \text{in} \  \mathbb{R}^2 \backslash \overline{B_R},\quad  w^s_{\alpha\alpha} |_{\partial B_R}\in H^{-1/2}(\partial B_R).
	   \end{equation}
		
	\end{corollary}

	\begin{proof}
	Without losing generality, we only consider the case of $\alpha = p$. The case of $\alpha = s$ can be proceeded in a similar manner. 
Since
	\begin{equation*}
		\lambda_{B_R,p}^{(n)} = |\text{Re}(\eta_{B_R,p}^{(n)})| + |\text{Im}(\eta_{B_R,p}^{(n)})|,
	\end{equation*}	
we deduce from \eqref{fbr-eigpp} and the definition  $\Sigma_n:=\mathcal{H}_n^{-1} \mathcal{J}_n$ that	\begin{equation}
		\lambda_{B_R,p}^{(n)} = \sqrt{\frac{\pi}{ k_p}} \frac{t_p^{2n}}{2^{2n-1} (n-1)! (n-2)! (t_p^2 + t_s^2)}\left(1+\mathcal{O}(\frac{1}{n})\right), \ 
		\varphi_{B_R,p}^{(n)} = \text{e}^{in\theta_x}.
	\end{equation}
By the Jacobi-Anger expansion (see e.g.,\cite{kress.1998}), a Sommerfeld radiating solution $w^s_{pp}$ to the Helmholtz equation in $|x|>R$  can be expanded into the series
    \begin{equation} 
		w^s_{pp}(x) = \sum_{n \in \mathbb{Z}} D_{n} H_{n}^{(1)}(k_p |x|) \text{e}^{i n \theta_x},\quad |x|>R,
    	\quad x=(|x|,\theta_x),
    \end{equation}
with the far-field pattern given by (see \cite[(3.82)]{kress.1998})
    \begin{equation}
    w^{\infty}_{pp}(\hat{x}) = \sum_{n \in \mathbb{Z}} D_{n} C_{n,p} \text{e}^{in \theta_x}, \quad C_{n,p}:=\sqrt{\frac{2}{k_p \pi}} \text{e}^{-i (\frac{n\pi}{2}+\frac{\pi}{4})}.
    \end{equation}
	Hence,
	\begin{equation}\label{IBRP}
		\begin{split}
			I^{(p)}(B_R) &= \sum_{n \in \mathbb{Z}} \frac{\left|\left\langle \sum_{m \in \mathbb{Z}} D_{m} C_{m,p} \text{e}^{im \theta_x}, \varphi^{(n)}_{B_R, p}\right\rangle_{\mathbb{S}}\right|^{2}}{\left|\lambda^{(n)}_{B_R, p}\right|}\\
			&= \sum_{n \in \mathbb{Z}} \frac{\left| 2\pi D_{n} C_{n,p}\right|^{2}}{\left|\lambda^{(n)}_{B_R, p}\right|}\\
			&= \sum_{n \in \mathbb{Z}} \sqrt{\frac{\pi}{k_p}} \frac{2^{2n+2} (n-1)! (n-2)! (t_p^2+t_s^2)}{t_p^{2n}}|D_n|^2 \left(1+\mathcal{O}(\frac{1}{n})\right).
		\end{split}
	\end{equation}
	By the definition of $H^{-1/2}$-norm on $\partial B_R$, we get
	\begin{equation}\label{ws_bound}
		\begin{split}
		   ||w^s_{pp}(x)||^{2}_{H^{-1/2}(\partial B_R)} &= \sum_{n \in \mathbb{Z}} (1+n^2)^{-1/2} |D_n H^{(1)}_n(t_p)|^2\\
		   &=\sum_{n \in \mathbb{Z}} |D_{n}|^{2} \frac{2^{2n} (n-1)! (n-2)!}{\pi ^2 t_p^{2n}}\left(1+\mathcal{O}(\frac{1}{n})\right).
		\end{split}
	\end{equation}
	Obviously, the series \eqref{IBRP} and \eqref{ws_bound} have the same convergence. 
	On the other hand, following the proof of \cite[Theorem 2.15]{kress.1998}, it is not difficulty to prove that the boundedness of $||w^s_{pp}||_{H^{-1/2}(\partial B_R)}$  implies  that $w^s_{pp}$ is a radiating solution in $|x|>R$ with the far-field pattern $w^{\infty}_{pp}(\hat{x})$. This proves that $I^{(p)}(B_R)<\infty$ if and only if $w^s_{pp}$ is a radiating solution to the boundary value problem of the Helmholtz equation \eqref{Helmholtz-disk}, with the far-field pattern $w^{\infty}_{pp}(\hat{x})$. 
	\end{proof}

	\section{Imaging schemes with testing disks}\label{sec6}

	
Let $B_R(z)=z+B_R \coloneqq \{y\in \R^2 : y=z+x, x\in B_R\}$ be a rigid disk centered at $z\in \R$ with radius $R>0$. By the translation relations (see e.g., (2.13)-(2.16),\cite{liuxd2019}), we know
	\begin{equation}
		u_{B_R(z),\alpha\beta}^{\infty}(\hat{x}) = \text{e}^{-ik_{\alpha}z\cdot \hat{x}} \text{e}^{ik_{\beta}z\cdot d} u_{B_R,\alpha\beta}^{\infty}(\hat{x}),
	\end{equation}
	where $\alpha = p,s$ and $\beta = p,s$. Define the matrices
	\begin{equation*}
	\begin{split}
		&U_{B_R(z)}^{\infty}(\hat{x}) \coloneqq
		\begin{pmatrix}
			u^{\infty}_{B_R(z),pp}(\hat{x}) & u^{\infty}_{B_R(z),ps}(\hat{x}) \\ 
			u^{\infty}_{B_R(z),sp}(\hat{x}) & u^{\infty}_{B_R(z),ss}(\hat{x})
		\end{pmatrix},\\
	&	M_{d,z} \coloneqq 
		\begin{pmatrix}
			\text{e}^{-ik_pz\cdot d} & 0\\
			0 & \text{e}^{-ik_sz\cdot d}
		\end{pmatrix}, \quad
		M_{\hat{x},z} \coloneqq 
		\begin{pmatrix}
			\text{e}^{-ik_pz\cdot \hat{x}} & 0\\
			0 & \text{e}^{-ik_sz\cdot \hat{x}}
		\end{pmatrix}.
		\end{split}
	\end{equation*}
	Then, it holds that
	\begin{equation}
		U_{B_R(z)}^{\infty} = M_{d,z}^{-1}\; U^\infty_{B_R} \;M_{\hat{x},z}.
	\end{equation}
Using the previous relation, we obtain spectral systems for the operators $F_{\Omega}$, $F_{\Omega}^{(p)}$ and 	$F_{\Omega}^{(s)}$ with $\Omega=B_{R}(z)$ as follows.
	\begin{corollary}\label{C6.1} The eigenvalues $\lambda_{B_R(z), j}^{(n)}$ and the associated eigenfunctions $X^{(n)}_{B_R(z), j}$ of the far-field operator $F_{B_R(z)}$ are given by (e.g. (\ref{eigsys-fbr}))
\begin{equation*}
\begin{split}
&\lambda_{B_R(z), j}^{(n)}=\lambda_{B_R, j}^{(n)},\\
&X^{(n)}_{B_R(z), j}(\hat{x})=\frac{\left(\sqrt{k_p}\text{e}^{-ik_pz\cdot\hat{x}}\hat{x}+\sqrt{k_s}\sigma_j^{(n)}\text{e}^{-ik_sz\cdot \hat{x}}\hat{x}^{\bot}\right)}{\sqrt{1+|\sigma_j^{(n)}|^2}}
\text{e}^{in \theta_x}
\end{split}	
\end{equation*}	
	for $n\in \Z, \quad j=1,2$.
	Moreover, the spectral systems of the operators $F_{B_R(z)}^{(p)}$ and $F_{B_R(z)}^{(s)}$ take the form (e.g. (\ref{fbr-eigpp}))
	\begin{equation}
		\begin{split}
			\eta_{B_R(z),p}^{(n)} &= \sqrt{\frac{2}{\pi k_p}} \text{e}^{i\frac{\pi}{4}} i \Sigma_n(1,1),\quad
			\varphi_{B_R(z),p}^{(n)}(\hat{x})= \text{e}^{in\theta_x}\text{e}^{-ik_pz\cdot\hat{x}},\\
			\eta_{B_R(z),s}^{(n)} &= \sqrt{\frac{2}{\pi k_s}} \text{e}^{i\frac{\pi}{4}} i \Sigma_n(2,2),\quad
			\varphi_{B_R(z),s}^{(n)}(\hat{x}) = \text{e}^{in\theta_x}\text{e}^{-ik_sz\cdot\hat{x}}.
		\end{split}
	\end{equation}\end{corollary}
Furthermore, taking $\Omega = B_R(z)$, we can rewrite results of Theorems \ref{One-FM} and  \ref{One-FMpp} as
	\begin{theorem}\label{One-FMdisk}
		Define	
		\begin{equation}\label{svd-disk}
		\begin{split}
			&W(B_R(z)) := \sum_{n \in \mathbb{Z}}\sum_{j=1}^2 \frac{\left|\left\langle u_D^{\infty}, X^{(n)}_{B_R(z), j}\right\rangle_{\mathbb{S}}\right|^{2}}{\left|\lambda^{(n)}_{B_R(z), j}\right|},\\
			&W^{(\alpha)}(B_R(z)) := \sum_{n \in \mathbb{Z}} \frac{\left|\left\langle u_{D,\alpha \alpha}^{\infty}, \varphi^{(n)}_{B_R(z),\alpha}\right\rangle_{\mathbb{S}}\right|^{2}}{\left|\lambda^{(n)}_{B_R(z),\alpha}\right|},\quad \alpha=p, s
		\end{split}
		\end{equation}
		where 
		\[
		\lambda^{(n)}_{B_R(z),\alpha}
		= \left| \text{\rm Re}\left(\eta_{B_R(z),\alpha}^{(n)}\right)\right| +\left |\text{\rm Im}\left(\eta_{B_R(z),\alpha}^{(n)}\right)\right|	.			\]
		Then $W(B_R(z))<\infty$	if and only if $D \subseteq B_R(z)$ and the same conclusion applies to $W^{(\alpha)}(B_R(z))$. 
	\end{theorem}

			

Finally, we describe our imaging scheme for solving the inverse problems {\bf IP-F},
{\bf IP-P} and {\bf IP-S} stated at the end of Section \ref{sec:2}.
Let $D$ be a convex rigid polygon to be recovered from far-field data. The procedure consists of the following steps:
	\begin{itemize}
		\item Suppose that $B_R \supset D$ for some $R>0$ and collect the measurement data $u_D^{\infty}(\hat{x})$,  $u_{D,pp}^{\infty}(\hat{x})$ or 
		$u_{D,ss}^{\infty}(\hat{x})$ for all $\hat{x} \in \mathbb{S}$. Let $Q\supset D$ be our search/computational region for imaging $D$;
		\item Choose sampling centers $z_{n} \in \Gamma_{R}:= \{x:\ |x|=R\}$ for $n = 1,2,\cdots,N$ and
		choose sampling radii $h_m\in(0,2R)$ ($m=1,2,\cdots, M$) to get the spectral systems for the operators $F_{\Omega}$, $F_{\Omega}^{(p)}$ and 	$F_{\Omega}^{(s)}$ with $\Omega=B_{h_m}(z_n)$, $n=1,2,\cdots, N$, $m=1,2,\cdots, M$ (see \eqref{eigsys-fbr} and Corollary \ref{C6.1});
		\item For each $z_n\in \Gamma_R$, define the function over the grid points $x\in Q$ satsfying $h_{m+1}>|x-z_n|\geqslant h_m$ for some $m=1,2, \cdots, M$ by (see \eqref{svd-disk}):
		\begin{equation*}
		\mathcal{I}_{n}(x)=\left\{\begin{array}{lll}
		&[W(B_{h_m}(z_n))]^{-1}\qquad\mbox{for the inverse problem {\bf IP-F}};\\
		&[W^{(p)}(B_{h_m}(z_n))]^{-1}\quad\mbox{for the inverse problem {\bf IP-P}};\\		&[W^{(s)}(B_{h_m}(z_n))]^{-1}\quad\mbox{for the inverse problem {\bf IP-S}};		\end{array}\right.  
				\end{equation*}
		
	\item The imaging function for recovering $D$ is defined as $\mathcal{I}(x)= \sum_{n=1}^{N_z}\mathcal{I}_{n}(x)$,  where $x\in Q$ are the grid points. This can be considered as imaging function over $Q$ if the grids are sufficiently fine.
	\end{itemize}
We expect the values of the indicator function $\mathcal{I}$ for grid points $x\in D$ should be larger than those for $x\in Q\backslash\overline{D}$, because 
\begin{equation*}
\begin{split}
&[W(B_{h_m}(z_n))]^{-1}=0\quad\mbox{if}\quad h_m\leq\max_{y\in \partial D}|z_n-y|;\\
&[W(B_{h_m}(z_n))]^{-1}<\infty\quad\mbox{if}\quad h_m>\max_{y\in \partial D}|z_n-y|;\end{split}
\end{equation*}
and the same indicating behavior applies to $[W^{(\alpha)}(B_{h_m}(z_n))]^{-1}$, $\alpha=p,s$.	
\begin{remark}
In implementing the above scheme, the spectral data appeared in Theorem \ref{One-FMdisk} are all given explicitly by Corollary \ref{C6.1}. For each sampling disk $B_{h_m}(z_n)$ with $n=1,2,\cdots N, m=1,2,\cdots, M$,  they can be easily calculated and stored off-line before the inversion process. 
This is just the advantageous of using testing disks instead of other testing scatterers. 
\end{remark}

	\section{Acknowledgements}
This work was supported by NSFC 12071236 and NSAF U1930402.


\begin{thebibliography}{99}
		\bibitem{kress2002}C. Alves and R. Kress, {\em On the far-filed operator in elastic obstacle scattering}, IMA J. Appl. Math., 67 (2002): 1-21.
		
	\bibitem{Habib2015} H. Ammari, E. Bretin, J. Garnier, H. Kang, H. Lee and A. Wahab, {\em  
Mathematical Methods in Elasticity Imaging},
Volume 52, Princeton Series in Applied Mathematics,
Princeton University Press,  2015. 

		\bibitem{Arens} T. Arens, {\em Linear sampling methods for 2D inverse elastic wave scattering}, Inverse Problems, 17  (2001): 1445-1464.

		\bibitem{Bao2018} G. Bao, G. Hu, J. Sun and T. Yin, {\em Direct and inverse elastic scattering from anisotropic media}, J Math Pure Appl, 117 (2018): 263-301.


		\bibitem{Bla2018} E. Bl\aa sten and Y. Lin, {\em Radiating and non-radiating sources in elasticity}, Inverse Problems, 35 (2018): 015005.

		\bibitem{MC2005} M. Bonnet and A. Constantinescu, {\em Inverse problems in elasticity}, Inverse Problems, 21 (2015): 1-50.
		   
		\bibitem{Cha2003}A. Charalambopoulos, D. Gintides and K. Kiriaki, {\em The linear sampling method for non-absorbing penetrable elastic bodies}, Inverse Problems, 19 (2003): 549-561.

		\bibitem{Cha2007}A. Charalambopoulos, A. Kirsch, K. Anagnostopoulos, D. Gintides, and K. Kiriaki, {\em The factorization method in inverse elastic scattering from penetrable bodies}, Inverse Problems, 23 (2007): 27-51.

\bibitem{HC} Z. Cheng and G. Huang, 
{\em Reverse time migration for extended obstacles: Elastic waves (in Chinese)}, Science China Mathematics 45 (2015): 1103-1114.

		\bibitem{kress.1998}D. Colton and R. Kress,  \emph{Inverse Acoustic and Electromagnetic Scattering Theory}, 4th edition, Springer, Berlin, 2019.

	\bibitem{Das2008} S. Das, S. Banerjee and T. Kundu, {\em Elastic wave scattering in a solid half-space with a circular cylindrical hole using the Distributed Point Source Method}, INT J SOLIDS STRUCT, 45 (2008): 4498-4508.


\bibitem{el-hu19} J. Elschner and G. Hu, {\em Uniqueness and factorization method for inverse elastic scattering with a single incoming wave}, Inverse Problems, 35 (2019): 094002.
		
\bibitem{GS2012} D. Gintides and M. Sini, {\em Identification of obstacles using only the scattered P-waves or the scattered S-waves},
Inverse Problems Imaging, 6 (2012): 39-55.

		\bibitem{PGC1993}P. Hahner and G. C. Hsiao, {\em Uniqueness theorems in inverse obstacle scattering of elastic waves}, Inverse Problems, 9 (1993): 525.


		\bibitem{HuKirsch2012}G. Hu, A. Kirsch and M. Sini, {\em Some inverse problems arising from elastic scattering by rigid obstacles}, Inverse Problems, 29 (2012): 015009.

		\bibitem{I1999} M. Ikehata, {\em Enclosing a polygonal cavity in a two-dimensional bounded domain from Cauchy data}, Inverse Problems, 15 (1999): 1231-1241.

		\bibitem{II2009} M. Ikehata and H. Itou,   {\em Extracting the support function of a cavity in an isotropic elastic body from a single set of boundary data}, Inverse Problems, 25 (2009): 105005.

		\bibitem{liuxd2019}X. Ji and X. Liu, \emph{Inverse elastic scattering problems with phaseless far field data}, Inverse Problems, 35 (2019): 114004. 

		\bibitem{Ji2018}X. Ji, X. Liu, and Y. Xi, {\em Direct sampling methods for inverse elastic scattering problems}, Inverse Problems, 34 (2018): 035008.

		\bibitem{K2008} A. Kirsch and N. Grinberg, {\em The Factorization Method for Inverse Problems}, NewYork: Oxford University Press, 2008.	

		\bibitem{Ku1979} V. D. Kupradze, et. al., {\em Three-Dimensional Problems of the Mathematical Theory of Elasticity and Thermoelasticity}, Amsterdam: North-Holland, 1979.

		\bibitem{K2003}S. Kusiak, R. Potthast and J. Sylvester, {\em A range test for determining scatterers with unknown physical properties}, Inverse Problems, 19 (2003): 533-547.

		\bibitem{Li2016} P. Li, Y. Wang, Z. Wang and Y. Zhao, {\em Inverse obstacle scattering for elastic waves}, Inverse Problems, 32 (2016): 115018. 

		\bibitem{Lin2021} Y. Lin, G. Nakamura, R. Potthast and H. Wang, {\em Duality between range and no-response tests and its application for inverse problems}, Inverse Problems and Imaging, 15 (2021): 367-386.
		
		\bibitem{sun2019}J. Liu, X. Liu, and J. Sun, {\em Extended sampling method for inverse elastic scattering problems using one incident wave}, SIAM J. Imaging Sci., 12 (2019): 874-892.

		\bibitem{S2018} J. Liu and J. Sun, {\em Extended sampling method in inverse scattering}, Inverse Problems, 34 (2018): 085007.	

		\bibitem{L2003} D. R. Luke and R. Potthast, {\em The no response test - a sampling method for inverse scattering problems}, SIAM J. Appl. Math., 63 (2003): 1292-1312.

		\bibitem{M-H}G. Ma and G. Hu, {\em Factorization method with one plane wave: from model-driven and data-driven perspectives}, arXiv 2101.09664.

		\bibitem{NP2013} G. Nakamura and R. Potthast, {\em Inverse Modeling - an introduction to the theory and methods of inverse problems and data assimilation}, IOP Ebook Series,  2015.


\end{thebibliography}
\end{document}